\documentclass[12pt,a4paper,final]{amsart}
\usepackage{amssymb}
\usepackage{hyperref}
\usepackage[notcite, notref]{showkeys}
\usepackage[utf8]{inputenc}

\usepackage{color}
\theoremstyle{plain}
\newtheorem{theorem}{Theorem}[section]
\newtheorem{proposition}[theorem]{Proposition}
\newtheorem{corollary}[theorem]{Corollary}
\newtheorem{lemma}[theorem]{Lemma}

\newtheorem*{nnclaim}{Claim}

\theoremstyle{definition}
\newtheorem{definition}[theorem]{Definition}
\newtheorem{example}[theorem]{Example}
\newtheorem{remark}[theorem]{Remark}
\newtheorem{question}[theorem]{Question}
\newtheorem{conjecture}[theorem]{Conjecture}
\newtheorem{problem}[theorem]{Problem}
\newtheorem{notation}[theorem]{Notation}

\theoremstyle{remark}
\newtheorem{claim}{Claim}

\numberwithin{equation}{section}
\newcommand{\N}{\mathbb N}
\newcommand{\Z}{\mathbb Z}

\newcommand{\R}{\mathbb R}
\newcommand{\C}{\mathbb C}
\newcommand{\F}{\mathbb F}
\newcommand{\Hy}{\mathbb H}

\newcommand{\fg}{\mathfrak g}

\newcommand{\ft}{\mathfrak t}

\DeclareMathOperator{\SL}{SL}
\DeclareMathOperator{\Ot}{O}
\DeclareMathOperator{\SO}{SO}
\DeclareMathOperator{\SU}{SU}
\DeclareMathOperator{\Sp}{Sp}
\DeclareMathOperator{\Ut}{U}

\DeclareMathOperator{\Spin}{Spin}

\newcommand{\op}{\operatorname}

\newcommand{\ba}{\backslash}
\newcommand{\mi}{\mathrm{i}}
\newcommand{\mj}{\mathrm{j}}
\newcommand{\mk}{\mathrm{k}}

\DeclareMathOperator{\diag}{diag}

\newcommand{\innerdots}{\langle {\cdot},{\cdot}\rangle }

\DeclareMathOperator{\Ad}{Ad}

\DeclareMathOperator{\Cas}{Cas}

\DeclareMathOperator{\Spec}{Spec}
\DeclareMathOperator{\mult}{mult}
\newcommand{\norma}[1]{\|{#1}\|_1}

\DeclareMathOperator{\Iso}{Iso}

\DeclareMathOperator{\vol}{vol}

\newcommand{\lens}{\mathfrak L}
\newcommand{\unique}{\mathfrak L^{\bullet}}
\newcommand{\LMR}{L}%{\mathrm{LMR}}

\title[The spectral geometry of hyperbolic and spherical manifolds]{The spectral geometry of hyperbolic and spherical manifolds: analogies and open problems}

\author{Emilio~A.~Lauret}
\address{Instituto de Matemática (INMABB), Departamento de Matemática, Universidad Nacional del Sur (UNS)-CONICET, Bahía Blanca, Argentina.}
\email{emilio.lauret@uns.edu.ar}

\author{Benjamin Linowitz}
\address{Department of Mathematics\\ 
10 North Professor Street\\
Oberlin, OH 44074.}
\email{benjamin.linowitz@oberlin.edu}

\thanks{This research was supported by grants from FONCyT (BID-PICT-2018-02073 and BID-PICT-2019-2019-01054) and SGCYT--UNS. The second author is partially supported by NSF Grant Number DMS-1905437.}
\date{\today}

\begin{document}

\begin{abstract} 
The spectral geometry of negatively curved manifolds has received more attention than its positive curvature counterpart. In this paper we will survey a variety of spectral geometry results that are known to hold in the context of hyperbolic manifolds and discuss the extent to which analogous results hold in the setting of spherical manifolds. We conclude with a number of open problems.
\end{abstract}

\maketitle

\tableofcontents
	
\section{Introduction}
Let $(M,g)$ be a compact Riemannian manifold. The eigenvalues of the Laplace-Beltrami operator acting on the space $L^2(M,g)$ form a discrete subset of the non-negative real numbers in which every value occurs with a finite multiplicity. This collection of eigenvalues is called the spectrum of $(M,g)$ and is denoted by $\Spec(M,g)$. Two Riemannian manifolds are said to be isospectral if their spectra coincide.

Inverse Spectral Geometry studies to what extent the geometry and topology of $(M,g)$ are determined by $\Spec(M,g)$. It is well known, for example, that $\dim M$ and $\vol(M,g)$ are both spectral invariants; that is, their values are both determined by $\Spec(M,g)$. Isometry class is not a spectral invariant, however. Indeed, the literature is full of interesting examples of Riemannian manifolds that are isospectral but not isometric. Alluding to Kac's famous ``Can one hear the shape of a drum'' article \cite{Kac66}, spectral invariants are called audible. Part of the importance of examples of isospectral Riemannian manifolds is their ability to show that certain properties are inaudible. For a more detailed discussion we refer the reader to the survey
\cite{GordonSurvey} by Gordon.

Locally symmetric spaces arise frequently in the construction of isospectral manifolds. In fact, the first three classes of examples of isospectral manifolds were all locally symmetric spaces: flat tori by Milnor~\cite{Milnor64}, Riemann surfaces by Vignéras~\cite{Vigneras80}, and lens spaces by Ikeda~\cite{Ikeda80_isosp-lens}. 
Subsequently, compact locally symmetric spaces of non-compact type (that is, compact manifolds covered by non-compact symmetric spaces; e.g. compact hyperbolic manifolds) have attracted more attention than locally symmetric spaces of compact type (e.g. spherical space forms). 

The main goal of this article is to discuss possible extensions to the compact type setting of several results in inverse spectral geometry of locally symmetric spaces of non-compact type. We introduce each of these results individually in Subsections~\ref{subsec:highestvolume}--\ref{subsec:McReynolds} and address them in Sections~\ref{sec:highestvolume}--\ref{sec:McReynolds}, respectively.
The article ends in Section~\ref{sec:openquestions} with further open questions and problems.

Before discussing this paper's results, we introduce the main actors. A spherical space form is a Riemannian manifold of the form $S^d/\Gamma$ where $S^d$ denotes the $d$-dimensional sphere endowed with its constant sectional curvature one Riemannian metric and where $\Gamma$ is a discrete (hence finite) subgroup of $\Iso(S^d)$ acting freely on $S^d$. 
An important subclass of spherical space forms are those of odd dimension with $\Gamma$ cyclic. These spaces are called lens spaces. 
Lens spaces have long played an important role in inverse spectral geometry, beginning with Ikeda's aforementioned examples \cite{Ikeda80_3-dimI} of isospectral lens spaces. For more recent work on the spectral geometry of lens spaces, see e.g.\ \cite{LMR-SaoPaulo} and the references therein. Rather than working with manifolds, we will often consider instead (good) orbifolds. 
These spaces, when defined as above, though with the free-action condition omitted, will be called \emph{spherical orbifolds} and \emph{lens orbifolds} respectively. 
See \S\ref{sec:preliminaries} for more details.

\subsection{Eigenvalue equivalence} 

In 1911 Weyl derived an asymptotic expression for the sequence of Laplace eigenvalues of a compact Riemannian manifold $(M,g)$ which implied that $\vol(M,g)$ is a spectral invariant. In particular, isospectral manifolds necessarily have the same volume. Recall, however, that the spectrum of $(M,g)$ is the set of eigenvalues of the Laplace-Beltrami operator acting on $L^2(M,g)$, counted with multiplicity. It is therefore natural to ask whether Riemannian manifolds which have the same set of Laplace-Beltrami eigenvalues (disregarding multiplicities) necessarily have the same volume. Call two Riemannian manifolds eigenvalue equivalent if they have the same set of Laplace-Beltrami eigenvalues. The question is therefore whether eigenvalue equivalent manifolds must have the same volume. 

In 2007, Leininger, McReynolds, Neumann and Reid~\cite{LMNR} proved that this is false by constructing examples of hyperbolic $n$-manifolds (for any dimension $n\geq 2$) which are eigenvalue equivalent and whose volumes differ.

\begin{question}
Do there exist examples of eigenvalue equivalent spherical orbifolds whose volumes differ?
\end{question}

In Section \ref{sec:eigenvalueequivalence} we will answer this question in the affirmative in every dimension $n\geq 9$ by proving the following theorem.

\begin{theorem}
Let $n\geq 9$. There exist $n$-dimensional spherical orbifolds $M_1$ and $M_2$ which are eigenvalue equivalent yet whose volumes satisfy $\mathrm{vol}(M_2)=3\cdot \mathrm{vol}(M_1)$.
\end{theorem}

We also give a second construction of eigenvalue equivalent spherical orbifolds (lens spaces) whose volumes differ. This construction does not rely on the work of Leininger, McReynolds, Neumann and Reid~\cite{LMNR} but rather makes use of an explicit formula for the Laplace eigenvalues of an arbitrary lens space.

\begin{theorem}\label{thm:lenseigenvalueequivalent}
For every odd integer $d\geq3$, there exists an infinite family $\mathfrak L$ of $d$-dimensional lens spaces that are mutually eigenvalue equivalent and such that 
\begin{equation}
\sup_{L_1,L_2\in\mathfrak L} \frac{\vol(L_1)}{\vol(L_2)}=\infty
.
\end{equation}
\end{theorem}

We conclude Section~\ref{sec:eigenvalueequivalence} by showing that the eigenvalue spectrum (with multiplicities disregarded) cannot detect singularities in lens orbifolds (Example~\ref{ex:eigenvaluespectrumdoesnotdetectsingularities}) and that it does not determine the dimension (Theorem~\ref{thm:eigenvalueequivalent-diffdimensions}) of homogeneous Riemannian manifolds of compact type.

\subsection{Isospectral pairs of largest volume}\label{subsec:highestvolume} 

The first pairs of isospectral hyperbolic surfaces were constructed by Vign{\'e}ras~\cite{Vigneras80} and had enormous area. A decade later Buser~\cite{Buser92} used Sunada's method \cite{Sunada85}, a powerful method that can be used to construct isospectral Riemannian manifolds in many different contexts, in order to construct isospectral hyperbolic surfaces of genus $5$ and of genus $g$ for all  $g\geq 7$. Examples with genus $4$ and $6$ were later constructed by Brooks and Tse~\cite{BT}. There are no known examples of non-isometric isospectral hyperbolic surfaces with genus $2$ or $3$, and it is suspected that in genus $2$ such surfaces cannot exist. 

In the arithmetic realm, John Voight and the second author constructed pairs of non-isometric (strongly) isospectral 2-dimensional and 3-dimensional arithmetic hyperbolic orbifolds and manifolds of minimal volume among certain nice classes of arithmetic orbifolds (see \cite{LinowitzVoight15}). It is not known what the smallest area of a pair of arithmetic hyperbolic $2$-orbifolds is, but one may suspect that it is the area of a particularly simple pair of isospectral hyperbolic polygons found by Doyle and Rossetti~\cite[Section 2]{DR}.

In this paper we will discuss spherical analogs of the above results. 
The volume of a spherical orbifold $S^d/\Gamma$ is given by 
\begin{equation}
\vol(S^d/\Gamma) = \frac{\vol(S^d)}{|\Gamma|}. 
\end{equation} 
In particular, the volume of any $d$-dimensional spherical orbifold is bounded by above by $\vol(S^d)$. This makes it clear that the spherical analog of the problem of finding the isospectral hyperbolic manifolds of smallest volume is to find the isospectral spherical manifolds of largest volume.

\begin{question}\label{question:highestvolume-orbifold}
What is the largest volume of an isospectral pair of $d$-dimensional non-isometric spherical orbifolds?
\end{question}

A pair of almost conjugate (and non-conjugate) subgroups in $\SO(6)$ constructed by Rossetti, Schueth and Weilandt~\cite{RossettiSchuethWeilandt08}  immediately provides the following quite good lower bound. 

%The very first example of almost conjugate subgroups, given by Gassmann, are subgroups of the symmetric group in $6$ letters and induce almost conjugate subgroups of $\SO(6)$. By applying Sunada's method, these groups provide us with a nearly complete solution to Question \ref{question:highestvolume-orbifold}.

\begin{theorem}\label{thm1:highestvolume-orbifold}
The largest volume of an isospectral and non-isometric pair of spherical orbifolds of dimension $d\geq5$ is at most $\frac18\vol(S^d)$. 
\end{theorem}

%This answer for Question~\ref{question:highestvolume-orbifold} is \emph{currently} optimal in the sense that it is not known whether there exist isospectral spherical orbifolds of dimension $d\leq 4$.
%See Remark~\ref{rem:orbifold-dim<=4} for more advances in this problem. 

As in the hyperbolic setting, the difficulty increases when we consider manifolds instead of orbifolds.
In the spherical context, this situation is explained because the condition of acting freely on $S^{d}$ is a great obstruction for a finite subgroup of $\SO(d+1)$. 
Indeed, although every finite group can be embedded into an special orthogonal group, the classification of spherical space forms done by Wolf~\cite{Wolf-book} shows that the groups acting freely on spheres are very particular.

\begin{question}\label{question:highestvolume-manifold}
What is the largest volume of an isospectral and non-isometric pair of $d$-dimensional spherical space forms? 
\end{question}

It is well known that there do not exist isospectral pairs of spherical space forms when $d$ is even or when $d=3$.
For all other values of $d$ we have the following statement. 

\begin{theorem}\label{thm1:highestvolume-manifold}
If $n\geq 3$, then any pair of isospectral and non-isometric $(2n-1)$-dimensional spherical space forms of largest volume are lens spaces provided that
\begin{equation}\label{eq:congruences}
\begin{cases}
n\equiv 1&\pmod 4,\text{ or}\\
n\equiv 1,2,3&\pmod 5,\text{ or}\\
n\equiv 1,2,3,4&\pmod 6,\text{ or}\\
n\equiv 2,3,4,5,6&\pmod 8,\text{ or}\\
n\equiv 2,3,4,5,6,7&\pmod 9,\text{ or}\\
n\equiv 2,3,4,5,6,7,8,9 &\pmod{11}.
\end{cases}
\end{equation}
In particular, this holds for all $3\leq n\leq 1000$ with the sole exceptions of $n=144$ and $n=935$. 
\end{theorem}

The authors conjecture that the above statement in fact holds for all $n\geq3$. 
See Tables~\ref{table:lens} and \ref{table:existencelens} for an explicit upper bound for the volume of a pair of $(2n-1)$-dimensional isospectral and non-isometric lens spaces, for each $n$ satisfying \eqref{eq:congruences}.

\subsection{Finite part spectrum} \label{subsec:finitespectrum}

In some situations, only a finite part of the spectrum is necessary to determine isospectrality. 
This is the case for Riemann surfaces under some geometric obstructions.

\begin{theorem}[Buser, Courtois \cite{BuserCourtois90}] \label{thm1:BuserCourtois}
Given an integer $g\geq2$ and $\varepsilon>0$, there is $N=N(g,\varepsilon)$ such that two compact Riemann surfaces of genus $g$ and injectivity radius $\geq\varepsilon$ are isospectral if and only if they have the same first $N$ eigenvalues (counted with multiplicities). 
\end{theorem}

Dai and Wei~\cite{DaiWei94} obtained a nice extension valid for the moduli space of Einstein metrics under some geometric conditions. 

We discuss in Section~\ref{sec:finitespectrum} some extensions of Theorem~\ref{thm1:BuserCourtois} among quotients of compact symmetric spaces. 
Indeed, we will observe that a CROSS (Compact Rank One Symmetric Space) is a very adequate choice for this sort of question. 
The proofs will follow (more or less immediately) from Lie theoretical results in \cite{LM-strongmultonethm} by Miatello and the first named author; in fact, \cite[Rem.~3.8]{LM-repequiv2} predicted such a situation though without providing details. 

We realize simply connected CROSSes as quotients of compact Lie groups as follows:  
\begin{align}\label{eq:CROSSrealizations}
S^{n}&=\tfrac{\SO(n+1)}{\SO(n)}, &
P^n(\C)&=\tfrac{\SU(n+1)}{\op{S}(\Ut(n)\times\Ut(1))}, & P^n(\Hy)&=\tfrac{\Sp(n+1)}{\Sp(n)\times\Sp(1)}, &
P^2(\mathbb O)&=\tfrac{\op{F}_4}{\Spin(9)}. 
\end{align}
The only non-simply connected CROSSes are real projective spaces that we write $P^n(\R)=\SO(n+1)/\Ot(n)$. 
Note that $G$ acts almost effectively and by isometries on $X=G/K$ in every case. 

\begin{theorem}\label{thm1:finitespectrum}
Let $X$ be a compact rank one symmetric space realized as $G/K$ as in \eqref{eq:CROSSrealizations}. 
Given $\varepsilon>0$, there is $N=N(X,\varepsilon)$ such that, for $\Gamma_1,\Gamma_2$ finite subgroups of $G$ with $|\Gamma_i|^{-1}=\frac{\vol(\Gamma_i\ba X)}{\vol(X)} >\varepsilon$, the orbifolds $\Gamma_1\ba X$ and $\Gamma_2\ba X$ are isospectral if and only if they have the same first $N$ eigenvalues (counted with multiplicities). 
\end{theorem}

We note that the condition $|\Gamma_i|^{-1}>\varepsilon$ cannot be omitted from Theorem \ref{thm1:finitespectrum}, as Example~\ref{ex:omittingvolume} shows.

\subsection{Isospectral towers of lens spaces}

In order to state our results on isospectral towers we will need the following definitions.

\begin{definition}
A descending (respectively, ascending) tower of covers is a set of Riemannian manifolds $\{M_I\}$ indexed by a poset $\mathcal S$ such that if $I<J$, then there is a finite degree Riemannian covering $M_I\longrightarrow M_J$ (respectively, $M_J\longrightarrow M_I$). 
\end{definition}

Towers of Riemannian manifolds appear frequently in the literature. As an example, seminal work of Buser and Sarnak~\cite{BS} studied the growth of the systole along towers of arithmetic hyperbolic surfaces. This work was generalized to towers of arithmetic hyperbolic $3$-manifolds by Katz, Schaps and Vishne~\cite{KSV}, and to arbitrary arithmetic locally symmetric spaces by Lapan, Linowitz, and Meyer~\cite{LLM}.

\begin{definition}
If $\{M_I\}$ and $\{N_I\}$ are two towers of Riemannian manifolds indexed by a poset $\mathcal S$, then we say that $\{M_I\}$ and $\{N_I\}$ are a pair of isospectral towers if, for all $I$, the manifolds $M_I$ and $N_I$ are isospectral and not isometric.
\end{definition}

In \cite{McReynolds14}, McReynolds used a variant of Sunada's method in order to construct pairs of isospectral (ascending) towers of Riemannian manifolds comprised of manifold quotients of symmetric spaces associated to non-compact Lie groups. Additional examples, which are not derivable from Sunada's method, were obtained by Linowitz in \cite{Linowitz12}, in the context of quotients of products $\mathbf{H}_2^a\mathbf{H}_3^b$ of hyperbolic upper-half planes and upper-half spaces by discrete groups of isometries obtained via orders in quaternion algebras.

The following is the natural spherical analog of the aforementioned results.

\begin{question}
Do there exist isospectral towers of spherical manifolds?
\end{question}

In Section \ref{sec:towers} we will completely answer this question by constructing isospectral towers of lens spaces.

\begin{theorem}
There exist infinitely many pairs of descending isospectral towers of lens spaces in every odd dimension $n\geq 5$.
\end{theorem}

\subsection{Isospectrality between quotients of symmetric spaces} \label{subsec:McReynolds}

In \cite{BGG}, Brooks, Gornet, and Gustafson used Sunada's method in order to construct arbitrarily large families of pairwise isospectral, non-isometric Riemann surfaces. More generally, Spatzier~\cite{Spatzier89} proved that every compact irreducible locally symmetric space admits a pair of isospectral finite covers provided its universal cover $X=G/K$ satisfies that $G$ is of type $\textup{A}_{n}$ for $n\geq26$, $\textup{B}_{n}$ or $\textup{D}_{n}$ for $n\geq13$, or $\textup{C}_{n}$ for $n\geq27$.    After several more examples appeared, McReynolds~\cite[Cor.~1.2]{McReynolds14} established the following result.

\begin{theorem}[McReynolds]\label{thm1:McReynolds}
Every non-compact irreducible simply connected symmetric space $X$ admits isospectral and non-isometric locally symmetric spaces covered by $X$.
\end{theorem}

Actually, McReynolds proved that for any $n$ there exist $n$ closed isospectral non-isometric manifolds with universal cover $X$.

The aim of Section~\ref{sec:McReynolds} is to discuss the analogous situation for locally symmetric spaces of compact type, which turns out to be very different from the non-compact type setting described above. For instance, there are many compact irreducible symmetric spaces that do not cover any manifold at all. 
Therefore, a natural question is the following. 

\begin{question}\label{prob:symspaces-manifolds}
Which compact simply connected irreducible symmetric spaces cover isospectral and non-isometric manifolds?
\end{question}

One of our results is the following. 

\begin{theorem}\label{thm1:symectricspaces}
Let $G$ be a compact connected simple Lie group of dimension at least $4$ and let $g_0$ be a bi-invariant metric on $G$ (thus $(G,g_0)$ is isometric to an irreducible symmetric space of group type). 
Then $(G,g_0)$ covers 
%Every simply connected compact irreducible symmetric spaces of group type of dimension at least $4$ cover 
isospectral and non-isometric manifolds, with the possible exception of $G=\SU(3)$, $\Sp(2)$, and $\op{G}_2$. % $\frac{\SU(3)\times\SU(3)}{\SU(3)}$, $\frac{\Sp(2)\times\Sp(2)}{\Sp(2)}$, and $\frac{\op{G}_2\times\op{G}_2}{\op{G}_2}$.
\end{theorem}

The situation for compact irreducible symmetric spaces of non-group type is less unified, so a detailed description in this case is postponed to Section~\ref{sec:McReynolds}.

\subsection*{Acknowledgments}
The authors wishes to express their thanks to the referee for several helpful comments. 
Furthermore, they are indebted to Loren Spice by a very useful answer in Mathoverflow to a question made by the first named author.

\section{Preliminaries}\label{sec:preliminaries}

In this section we introduce fundamental tools that will be used throughout this article. 

\subsection{Spherical space forms}
We consider the $d$-dimensional sphere $S^{d}$ with its Riemannian metric of constant sectional curvature one. 
Its isometry group $\Iso(S^{d})$ is given by $\Ot(d+1)$ via multiplication at the left, where the elements in $S^{d}$ are considered as vertical vectors with $d+1$ entries and Euclidean norm one. 
Similarly, the subgroup of preserving-orientation isometries satisfies $\Iso^0(S^{d})=\SO(d+1)$.

A \emph{spherical space form} is a compact Riemannian manifold with constant sectional curvature.  
Throughout this paper we will assume that the sectional curvature is one, unless explicitly stated otherwise, so that any $d$-dimensional spherical space form is covered by $S^{d}$. 
More precisely, a spherical space form is isometric to $S^{d}/\Gamma$, where $\Gamma$ is a discrete (hence finite) subgroup of $\Ot(d+1)$ acting freely on $S^{d}$. The manifold $S^{d}/\Gamma$ is orientable if and only if $\Gamma\subset\SO(d+1)$.

A non-trivial element $\gamma$ in $\Ot(d+1)$ acts freely on $S^{d}$ (i.e.\ if $\gamma\cdot x=x$ for some $x\in S^{d}$, then $\gamma=\op{I}_{d+1}$) if $+1$ is not an eigenvalue of $\gamma$. 
One can see that the only even-dimensional spherical space forms are spheres $S^{2n}$ and real projective spaces $P^{2n}(\R)$. 
The latter is not orientable since it is the quotient by $\{\pm \op{I}_{2n+1}\}$, which is not contained in $\SO(2n+1)$.
Every odd-dimensional spherical space form is orientable. 

The classification of spherical space forms was obtained by Wolf~\cite{Wolf-book} following techniques due to Vincent (see \cite[\S5.1]{Wolf-book} for a clear explanation). 
Although we will use it several times along this article, we omit the statement because it is quite technical and long, though any reference will be to a precise place in \cite{Wolf-book} (or \cite[\S3--5]{Wolf01}). 

Now, let $\Gamma$ be an arbitrary finite subgroup of $\Ot(d+1)$. 
The quotient $S^{d}/\Gamma$ has a structure of a (Riemannian) good orbifold and is called a \emph{spherical orbifold}. 
Two spherical orbifolds $S^{d}/\Gamma_1$ and $S^{d}/\Gamma_2$ are isometric if and only if $\Gamma_1$ and $\Gamma_2$ are conjugate in $\Ot(d+1)$.

\subsection{Spectral generating functions}
As usual, the \emph{spectrum} of a (compact) Riemannian manifold $M$ is the spectrum of its associated Laplace-Beltrami operator; that is, the the spectrum of $M$ is the collection of Laplace-Beltrami eigenvalues counted with multiplicity. We denote this spectrum by $\Spec(M,g)$.

The spectra of spheres have been known for a long time. 
The eigenfunctions are precisely the spherical harmonics restricted to the corresponding sphere. 
More precisely, by denoting by $\mathcal H_k$ the space of harmonic (w.r.t.\ the euclidean Laplacian on $\R^{d+1}$) homogeneous complex polynomials of degree $k$ in $d+1$ variables, the restriction of $f\in\mathcal H_k$ to $S^{d}$ is an eigenfunction of the Laplacian $\Delta$ of $S^{d}$ with eigenvalue $\lambda_k:=k(k+d-1)$. 
Moreover, since $L^2(S^{d})\simeq\bigoplus_{k\geq0} \mathcal H_k$ because polynomials are dense in the space of continuous functions, $\Spec(S^{d})$ is given by the multiset
\begin{equation}
\Spec(S^{d})=\Big\{\!\!\Big\{ \underbrace{\lambda_k,\dots,\lambda_k}_{\dim\mathcal H_k} \mid k\geq0 \Big\}\!\!\Big\}. 
\end{equation}

There is a well defined Laplacian on every good orbifold (see e.g.\ \cite{Gordon12-orbifold}).
In this article, we will assume for simplicity that the good orbifold is of the form $M/\Gamma$ with $M$ a compact Riemannian manifold and $\Gamma$ is a group acting effectively and by isometries on $M$. 
Thus, every $\Gamma$-invariant eigenfunction on $M$ descends to an eigenfunction on $M/\Gamma$ with the same eigenvalue, and moreover, every eigenfunction on $M$ is of this form.  

We now apply the above paragraph to our case of interest. 
Let $\Gamma$ be a finite subgroup of $\Ot(d+1)$.
Then $\Gamma$ acts on $\mathcal H_k$ by $(\gamma\cdot f)(x)=f(\gamma^{-1} x)$; we denote by $\mathcal H_k^\Gamma$ the subspace of $\Gamma$-invariant elements of $\mathcal H_k$. 
Note that $f\in\mathcal H_k$ is $\Gamma$-invariant considered as a function on $S^{d}$ if and only if $f\in\mathcal H_k^\Gamma$. 
We obtain that 
\begin{equation}\label{eq:Spec(S^d/Gamma)}
\Spec(S^{d}/\Gamma) =\Big\{\!\!\Big\{ \underbrace{\lambda_k,\dots,\lambda_k}_{\dim\mathcal H_k^\Gamma} \mid k\geq0 \Big\}\!\!\Big\}. 
\end{equation}

Ikeda was the first to consider inverse spectral problems for spherical space forms. 
His main tool was the generating function associated to a spherical orbifold $S^{d}/\Gamma$ given by 
\begin{equation}\label{eq:F_Gamma-definition}
F_{\Gamma}(z):= \sum_{k\geq0} \dim\mathcal H_k^\Gamma\, z^k. 
\end{equation}
Clearly, two spherical orbifolds $S^{d}/\Gamma_1,S^{d}/\Gamma_2$ are isospectral (i.e.\ $\Spec(S^{d}/\Gamma_1)=\Spec(S^{d}/\Gamma_2)$) if and only if $F_{\Gamma_1}(z)=F_{\Gamma_2}(z)$. 

Ikeda~\cite[Thm.~2.2]{Ikeda80_3-dimI} proved that 
\begin{equation}\label{eq:F_Gamma-expression}
F_{\Gamma}(z) = \frac{1-z^2}{|\Gamma|} \sum_{\gamma\in\Gamma} \frac{1}{\det(\op{I}_{d}-\gamma z)}, 
\end{equation}
where $\det(\op{I}_{d}-\gamma z) =\prod_{\lambda\in\Spec(\gamma)} (1-\lambda z)$. We note that here, $\Spec(\gamma)$ denotes the set of eigenvalues of $\gamma$, counted with multiplicities. As an immediate consequence, he obtain the following result that can be considered as a precursor of Sunada's method (\cite[Cor.~2.3]{Ikeda80_3-dimI}):
\begin{quote}
\it 
Let $\Gamma_1$ and $\Gamma_2$ be finite subgroups of $\Ot(d+1)$. 
If there is a bijection $\phi:\Gamma_1\to\Gamma_2$ satisfying that $\Spec(\gamma)=\Spec(\phi(\gamma))$ for all $\gamma\in\Gamma_1$, then the spherical orbifolds $S^{d}/\Gamma_1$ and $S^{d}/\Gamma_2$ are isospectral. 
\end{quote}
Ikeda constructed examples of isospectral spherical space forms via this result in \cite{Ikeda83}.

We next show a more general version obtained by Wolf~\cite[Cor.~2.13]{Wolf01}, which replaces the condition involving the spectra of the matrices by the following well-known notion.
\begin{definition}\label{def:almostconjugate}
Two subgroups $\Gamma_1,\Gamma_2$ of $G$ are called {\it almost conjugate} if there is a bijection $\phi:\Gamma_1\to \Gamma_2$ such that $h$ and $\phi(h)$ are conjugate in $G$ for all $h\in \Gamma_1$. 
\end{definition}

\begin{theorem}\label{thm:Sunadamethod}
Let $M$ be a compact Riemannian manifold. 
If $\Gamma_1$ and $\Gamma_2$ are almost conjugate finite subgroups of $\Iso(M)$, then the Riemannian good orbifolds $M/\Gamma_1$ and $M/\Gamma_2$ are strongly isospectral. 
\end{theorem}

Strongly isospectral manifolds satisfy the condition that, for any natural vector bundle, the natural strongly elliptic differential operators acting on square integrable sections of the corresponding vector bundles are isospectral, that is, they have the same spectra. 
Instances of these natural differential operators are the Laplace-Beltrami operator, the Hodge-Laplace operator acting on $p$-forms, the Lichnerowicz Laplacian acting on (symmetric) $k$-tensors, etc. 
For any of these operators, its spectrum on a good orbifold $M/\Gamma$ is obtained in a similar way as it was done for the Laplace-Beltrami operator above, namely, the eigensections on $M/\Gamma$ come from $\Gamma$-eigensections on $M$.

\subsection{Lens spaces}\label{subsec:lensspaces}
We now focus on lens spaces, which are quotients of odd-dimensional spheres by cyclic groups under free actions. 
Although lens spaces are topological spaces, we endow them with the constant sectional curvature one Riemannian metric so that they are spherical space forms. 
In fact, lens spaces are the odd-dimensional spherical space forms with cyclic fundamental group. 
Any of them is isometric to one of the following: 
for $q\in\N$ and $s=(s_1,\dots,s_n)\in\Z^n$ with $\gcd(q,s_i)=1$ for all $i$, we set $L(q;s)=S^{2n-1}/\Gamma_{q;s}$, where $\Gamma_{q;s}$ is the group generated by 
\begin{equation}\label{eq:gamma_qs}
\gamma_{q;s}:=
\begin{pmatrix}
R(\tfrac{2\pi s_1}{q}) \\
& \ddots \\
&&R(\tfrac{2\pi s_n}{q}) 
\end{pmatrix},
\qquad \text{where}\ 
R(\theta) = 
\begin{pmatrix}
\cos\theta & \sin\theta\\
-\sin\theta & \cos\theta
\end{pmatrix}
. 
\end{equation}
Sometimes we will write $L(q;s_1,\dots,s_n)=L(q;s)$. 

The condition $\gcd(q,s_i)=1$ for all $i$ ensures that $\Gamma_{q;s}$ acts freely on $S^{2n-1}$. 
The quotient $L(q;s)=S^{2n-1}/\Gamma_{q;s}$ with $s\in\Z^n$ satisfying the weaker condition $\gcd(q,s_1,\dots,s_n)=1$ is called a \emph{lens orbifold}. 

\begin{proposition}\label{prop:lens-isometries}
Let $L=L(q;s)$ and $L'=L(q;s')$ be two lens orbifolds of dimension $2n-1$. 
The following assertions are equivalent: 
\begin{enumerate}
\item $L(q;s)$ and $L(q;s')$ are homeomorphic. 

\item $L(q;s)$ and $L(q;s')$ are diffeomorphic. 

\item $L(q;s)$ and $L(q;s')$ are isometric. 

\item There are $\sigma$ a permutation of $\{1,\dots,n\}$, $\epsilon_i\in\{\pm1\}$ for each $i=1,\dots,n$, and $t\in\Z$ prime to $k$ such that 
\begin{equation}
s_{\sigma(i)} \equiv \epsilon_i t s_i\pmod q
\qquad\text{for all }i=1,\dots,n. 
\end{equation}
\end{enumerate}
\end{proposition}

Concerning the spectrum of a lens orbifold $L:=L(q;s)$, it follows form \eqref{eq:F_Gamma-expression} that 
\begin{equation}\label{eq:F_L-expression}
F_{L}(z):=F_{\Gamma_{q,s}}(z) = \frac{1-z^2}{q} \sum_{h=0}^{q-1} \frac{1}{ (z-\xi_q^{hs_1})(z-\xi_q^{-hs_1})\dots (z-\xi_q^{hs_n})(z-\xi_q^{-hs_n})}, 
\end{equation}
where $\xi_q=e^{2\pi\mi/q}$. 
This expression first appeared in \cite[Thm.~3.2]{IkedaYamamoto79}.

An alternative expression (see \cite[Thm.~3.1]{LMR-SaoPaulo}) is given by 
\begin{equation}\label{eq:F_L-thetafunction}
F_L(z) = \frac{1}{(1-z^2)^{n-1}} \sum_{k\geq0} N_{\mathcal L}(k)z^k,
\end{equation}
where $\mathcal L$ is the associated \emph{congruence lattice} of $L(q;s)$ given by 
\begin{equation}\label{eq:congruencelattice}
\mathcal L(q;s):=\{(a_1,\dots,a_n)\in\Z^n: a_1s_1+\dots+a_ns_n\equiv 0\pmod q\}
\end{equation}
and 
\begin{equation}
N_{\mathcal L}(k)=\#\{(a_1,\dots,a_n)\in\Z^n: |a_1|+\dots+|a_n|=k\}
.
\end{equation}
By expanding the right hand side of \eqref{eq:F_L-thetafunction}, one obtains 
\begin{equation}\label{eq:dimH_k^Gamma-onenorm}
\dim \mathcal H_k^\Gamma=\sum_{r=0}^{\lfloor k/2\rfloor} \binom{r+n-2}{n-2} N_{\mathcal L}(k-2r), 
\end{equation}
which was established in \cite[Thm.~3.5]{LMR-onenorm}.

%%%%%%%%%%%%%%%%%%%%%%%%%%%
%%%%%%%%%%%%%%%%%%%%%%%%%%%
%%%%%%%%%%%%%%%%%%%%%%%%%%%
%%%%%%%%%%%%%%%%%%%%%%%%%%%
%%%%%%%%%%%%%%%%%%%%%%%%%%%
%%%%%%%%%%%%%%%%%%%%%%%%%%%
%%%%%%%%%%%%%%%%%%%%%%%%%%%
%%%%%%%%%%%%%%%%%%%%%%%%%%%
%%%%%%%%%%%%%%%%%%%%%%%%%%%
%%%%%%%%%%%%%%%%%%%%%%%%%%%
%%%%%%%%%%%%%%%%%%%%%%%%%%%
%%%%%%%%%%%%%%%%%%%%%%%%%%%
%%%%%%%%%%%%%%%%%%%%%%%%%%%
%%%%%%%%%%%%%%%%%%%%%%%%%%%

%%%%%%%%%%%%%%%%%%%%%%%%%%%
%%%%%%%%%%%%%%%%%%%%%%%%%%%
%%%%%%%%%%%%%%%%%%%%%%%%%%%
%%%%%%%%%%%%%%%%%%%%%%%%%%%
%%%%%%%%%%%%%%%%%%%%%%%%%%%
%%%%%%%%%%%%%%%%%%%%%%%%%%%
%%%%%%%%%%%%%%%%%%%%%%%%%%%
%%%%%%%%%%%%%%%%%%%%%%%%%%%
%%%%%%%%%%%%%%%%%%%%%%%%%%%
%%%%%%%%%%%%%%%%%%%%%%%%%%%
%%%%%%%%%%%%%%%%%%%%%%%%%%%
%%%%%%%%%%%%%%%%%%%%%%%%%%%
%%%%%%%%%%%%%%%%%%%%%%%%%%%
%%%%%%%%%%%%%%%%%%%%%%%%%%%
%%%%%%%%%%%%%%%%%%%%%%%%%%%
%%%%%%%%%%%%%%%%%%%%%%%%%%%

\section{Eigenvalue equivalence} \label{sec:eigenvalueequivalence}

	The {\it eigenvalue spectrum} of a closed Riemannian orbifold is its set of eigenvalues for the Laplace-Beltrami operator, ignoring multiplicities. Two orbifolds are said to be {\it eigenvalue equivalent} if their eigenvalue spectra coincide.
	
	We will exhibit two different constructions of eigenvalue equivalent spherical orbifolds with different volume. 
	The first one will closely follow the work of Leininger, McReynolds, Neumann, and Reid~\cite{LMNR}, who considered the analogous problem for hyperbolic $n$-manifolds.
	The second one describes the eigenvalue spectrum of an arbitrary lens space making use of the approach in \cite{LMR-onenorm}, which is summarized in \cite[\S2--4]{LMR-SaoPaulo}. 
We conclude this section with examples of eigenvalue equivalent Riemannian manifolds of different dimensions.

\subsection{Eigenvalue equivalent spherical orbifolds}

\begin{definition}
Let $G$ be a finite group. Two subgroups $H$ and $K$ of $G$ are fixed point equivalent if, for any finite dimensional complex representation $\rho$ of $G$, the restriction $\rho\vert_H$ has a nontrivial fixed vector if and only if $\rho\vert_K$ does.
\end{definition}

Our interest in fixed point equivalent subgroups of finite groups is the following refinement of Sunada's method \cite[Theorem 2.6]{LMNR} (note that almost conjugate subgroups as in Definition~\ref{def:almostconjugate} are always fixed point equivalent \cite[Proposition 2.4]{LMNR}).

\begin{theorem}\label{sunadarefinement}
Let $H$ and $K$ be fixed point equivalent subgroups of a finite group $G$. If $M$ is a compact Riemannian manifold and $\pi_1(M)$ admits a surjective homomorphism onto $G$, then the covers $M_H$ and $M_K$ associated to the pullbacks in $\pi_1(M)$ of $H$ and $K$ have the same sets of eigenvalues of the Laplace-Beltrami operator.
\end{theorem}

We now prove the existence of spherical orbifolds whose sets of eigenvalues of the Laplace-Beltrami operator coincide yet whose volumes are different.

\begin{theorem}
Let $n\geq 9$. There exist $n$-dimensional spherical orbifolds $M_1$ and $M_2$ which are eigenvalue equivalent yet whose volumes satisfy $\mathrm{vol}(M_2)=3\cdot \mathrm{vol}(M_1)$.
\end{theorem}
\begin{proof}
Theorem 3.2 of \cite{LMNR} shows that there exist subgroups $K< H$ of $\mathrm{PSL}_2(\mathbb Z/9\mathbb Z)$ with $[H:K]=3$ and which are fixed point equivalent. An easy computation is SAGE shows that $\mathrm{PSL}_2(\mathbb Z/9\mathbb Z)$ is isomorphic to the subgroup $G=\langle (3,9,4,6)(5,10,8,7), (1,2,4)(5,6,8)(7,9,10)\rangle$ of the permutation group $S_{10}$. Identifying $S_{10}$ with the set of $10\times 10$ permutation matrices in $\Ot(10)$ allows us to identify $G$ with a subgroup of $\Ot(10)<\Ot(n+1)$. We may therefore assume that $G< \Ot(n+1)$. Let $M=S^{n}/G$ and $M_1, M_2$ be the covers of $M$ associated to the pullbacks in $G$ of $H$ and $K$ along the isomorphism $G\to \mathrm{PSL}_2(\mathbb Z/9\mathbb Z)$. That $M_1$ and $M_2$ have the desired properties now follows from Theorem \ref{sunadarefinement}.
\end{proof}

\subsection{Eigenvalue equivalent lens spaces}

From the description \eqref{eq:Spec(S^d/Gamma)} of the spectrum of the Laplace-Beltrami operator of a spherical orbifold $S^{2n-1}/\Gamma$, it follows immediately that the eigenvalue spectrum of $S^{2n-1}/\Gamma$ is given by 
\begin{equation}\label{eq:eigenvaluespectrum}
\mathcal E(S^{2n-1}/\Gamma):= \{ \lambda_k=k(k+2n-2) : k\in\N\text{ and }\dim\mathcal H_k^\Gamma>0 \},
\end{equation}
where $\mathcal H_k$ denotes the space of harmonic homogeneous complex polynomials of degree $k$ in $2n$ variables. 
We next focus in the particular case of lens spaces introduced in \S\ref{subsec:lensspaces}. 

\begin{theorem}\label{thm:eigenvaluespectrumlens}
The eigenvalue spectrum of a lens space $L=L(q;s_1,\dots,s_n)$  is given by 
\begin{equation}
\mathcal E(L)=\{\lambda_{2k}:k\in\N_0\} \cup 
\{\lambda_{k_0(L)+2k}:k\in\N_0\},
\end{equation}
where $\lambda_{k}=k(k+2n-2)$ and
\begin{equation}\label{eq:k_0}
k_0(L)=\min\left\{
k\in\N: 
\begin{array}{l}
\text{$k$ is odd and there is }(a_1,\dots,a_n)\in\Z^n \text{ such that}
\\
|a_1|+\dots+|a_n|=k \text{ and } 
a_1s_1+\dots+a_ns_n\equiv0\pmod q 
\end{array}
\right\}.
\end{equation}
\end{theorem}

\begin{proof}
Write $L=L(q;s_1,\dots,s_n)=S^{2n-1}/\Gamma$, with $s=(s_1,\dots,s_n)$ and $\Gamma=\Gamma_{q,s}$.
Recall from \eqref{eq:congruencelattice} that its associated congruence lattice is given by 
\begin{equation}
\mathcal L:=\mathcal L(q;s):=\{(a_1,\dots,a_n)\in\Z^n: a_1s_1+\dots+a_ns_n\equiv 0\pmod q\}.
\end{equation}

Besides \eqref{eq:eigenvaluespectrum}, the main tool in the proof will be the formula for $\dim \mathcal H_k^\Gamma$ given in \eqref{eq:dimH_k^Gamma-onenorm}:
\begin{equation}\label{eq:dimH_k^Gamma3}
\dim \mathcal H_k^\Gamma=\sum_{r=0}^{\lfloor k/2\rfloor} \binom{r+n-2}{n-2} N_{\mathcal L}(k-2r),
\end{equation}
where $N_{\mathcal L}(k)=\#\{\mu\in\mathcal L: \norma{\mu}=k\}$. 
Notice $\dim \mathcal H_k^\Gamma>0$ if and only if $N_{\mathcal L}(k-2r)>0$ for some $r\in \{0,1,\dots,\lfloor k/2\rfloor\}$. 

Suppose $k\in\N_0$ is even.
By setting $r=k/2$, we obtain that $N_{\mathcal L}(k-2r)=N_{\mathcal L}(0)=1$ since $(0,\dots,0)$ is clearly the only element in $\mathcal L$ with one-norm equal to $0$.
We conclude that $\lambda_k\in\mathcal E(L)$ for all $k$ even.

We now assume that $k\in\N$ is odd. 
The odd positive integer $k_0:=k_0(L)$ given in \eqref{eq:k_0} satisfies $N_{\mathcal L}(k_0)>0$. 
Hence $\lambda_k\in\mathcal E(L)$ for every odd integer $k\geq k_0$ by setting $r=\frac{k-k_0}{2}$ in \eqref{eq:dimH_k^Gamma3}.  
If $\lambda_{k}\in\mathcal E(L)$ for some odd integer $k<k_0$, then $N_{\mathcal L}(k_1)>0$ for some odd integer $k_1\leq k$ by \eqref{eq:dimH_k^Gamma3}, which implies the contradiction $k_1\geq k_0$. 
This concludes the proof. 
\end{proof}

Theorem~\ref{thm:eigenvaluespectrumlens} ensures that the eigenvalue spectrum of a lens space $L$ depends only on $k_0(L)$. 
This allows us to give curious examples of eigenvalue equivalent lens spaces.

\begin{proof}[Proof of Theorem~\ref{thm:lenseigenvalueequivalent}]
Write $d=2n-1$ for some $n\in\N$ with $n\geq2$.  
Let $\mathfrak L(n,q)$ denote the isometry classes of $(2n-1)$-dimensional lens spaces with fundamental group of order $q$. 
For a positive odd integer $k$, we set 
\begin{equation}
\mathfrak L(n,q)_{k}=\{L\in\mathfrak L(n,q): k_0(L)=k\}. 
\end{equation}
It follows immediately from Theorem~\ref{thm:eigenvaluespectrumlens} that the elements in 
\begin{equation*}
\mathfrak L^*(n)_k:=\bigcup_{q\geq3}\mathfrak L(n,q)_{k}
\end{equation*}
are pairwise non-isometric and mutually eigenvalue equivalent. 
It remains to show that $\mathfrak L^*(n)_{k}$ is infinite for some $k$. 

We pick $k=3$. 
One has that $L_q:=L(q;1,\dots,1,2)\in\mathfrak L(n,q)_3$ for any positive integer $q\geq3$.
In fact, $\mu:=(2,0,\dots,0,-1)\in\mathcal L(q;1,\dots,1,2)$ satisfies $\norma{\mu}=3$, and it is evident that no element of one-norm $1$ is contained in $\mathcal L(q;1,\dots,1,2)$.

The assertion about the volume follows immediately from 
\begin{equation}
\frac{\vol(L_{q'})}{\vol(L_{q})}
=\frac{\vol(S^{2n-1})}{q} \frac{q'}{\vol(S^{2n-1})}
=\frac{q'}{q}
\end{equation}
by taking $q=3$ and $q'$ arbitrarily large. 

\end{proof}

Theorem~\ref{thm:eigenvaluespectrumlens} also allows us to show that the eigenvalue spectrum does not detect singularities.

\begin{example}\label{ex:eigenvaluespectrumdoesnotdetectsingularities}
For positive integers $q,q'\geq3$, the lens orbifold $L(2q;1,2,q)$, which has non-trivial singularities, is eigenvalue equivalent to the lens space $L(q';1,1,2)$ since 
\begin{equation*}
k_0\big(L(2q;1,2,q)\big)=k_0\big(L(q';1,1,2)\big)=3.
\end{equation*} 
\end{example}

\subsection{Eigenvalue equivalent manifolds with different dimensions}\label{subsec:eigenvalueequivdimensions}

It is known that in general, dimension is not determined by the eigenvalue spectrum.
For instance, Miatello and Rossetti~\cite[Ex.~3.8]{MR-length} have observed that the eigenvalue spectrum of the square torus $\R^n/\Z^n$ is $\Z_{\geq0}$ for all $n\geq4$.

We now give examples of normal homogeneous Riemannian manifolds having the same eigenvalue spectrum but distinct dimensions. 
A particularly simple example is $\big(S^2,\tfrac{1}{4} \, g_{\text{round}}\big)$ and $\big(P^3(\R), g_{\text{round}}\big)$, where $g_{\text{round}}$ denotes in both cases the round metric with constant sectional curvature one. 
Indeed, on the one hand, the eigenvalues of $\big(S^2,\tfrac{1}{4} \, g_{\text{round}}\big)$, which is isometric to the sphere in $\R^3$ of radius $\frac12$ endowed with the Riemannian metric induced by the canonical Euclidean metric in $\R^3$, are given by $4$-times the eigenvalues of $\big(S^2,g_{\text{round}}\big)$, that is, 
\begin{equation}
\left\{ 4k(k+1):k\in\N_0 \right\}.
\end{equation} 
On the other hand, the eigenvalues of the $3$-dimensional real projective space $\big(P^3(\R), g_{\text{round}}\big)$ with constant sectional curvature one are the same as those of its $2$-cover $\big(S^3, g_{\text{round}}\big)$ associated to odd-dimensional representations, that is,
\begin{equation}
\left\{ k(k+2):k\in 2\N_0 \right\}.
\end{equation} 
Of course, the multiplicities do not match since $4k(k+1)$ has multiplicity $k+1$ in $\big(S^2,g_{\text{round}}\big)$ for any $k\in\N$, 
while $k(k+2)$ has multiplicity $k^2$ for any $k\in 2\N_0$. 

Both of the above manifolds are symmetric spaces. 
Our next result generalizes them and provides infinitely many examples of eigenvalue equivalent normal homogeneous Riemannian manifolds with different dimensions. 

\begin{theorem}\label{thm:eigenvalueequivalent-diffdimensions}
Let $G$ be a compact, connected, simply connected, and semisimple Lie group. 
The centerless Lie group $G/Z(G)$ endowed with the standard bi-invariant metric ($Z(G)$ is the center of $G$ and the metric is induced by the Killing form of $\fg$)
has the same eigenvalue spectrum as 
the standard metric on the full flag manifold $G/T$ ($T$ is a maximal torus in $G$ and the metric is induced by the Killing form of the Lie algebra $\fg$ of $G$).
Moreover, $\dim G/Z(G)>\dim G/T$.
\end{theorem}

\begin{proof}
The proof of this result makes use of objects from representation theory of compact Lie groups that have not appeared in the rest of the article. 
For conciseness, we will give the arguments without introducing several objects (e.g.\ weight lattice) and will instead  refer the interested reader to other references.

For an arbitrary compact connected Lie group $K$, the spectrum of the standard bi-invariant metric is given by 
\begin{equation}
\Big\{\!\! \Big\{ 
\underbrace{\Cas(\pi),\dots,\Cas(\pi)}_{(\dim V_\pi)^2}
%\Cas(\pi)\text{ with mult.\ } \dim V_\pi \dim V_\pi(0)
:\pi\in\widehat K
\Big\}\!\!\Big\}.
\end{equation}
Here, $\widehat K$ stands for the unitary dual of $K$ (the equivalent classes of irreducible representations of $K$), and
$\Cas(\pi)$ denotes the scalar for which the Casimir element of $\fg_\C$ acts on $V_\pi$. 
By the Highest Weight Theorem (see e.g.\ \cite[Thm.~7.34]{Sepanski}), $\widehat K$ is in correspondence with dominant analytically integral weights. 

In our case $K=G/Z(G)$, and since $K$ is centerless, \cite[Thm.~6.30(a)]{Sepanski} tells us that the lattice of analytically integral weights coincides with the root lattice $\mathcal R:=\bigoplus_{\alpha\in\Phi(\fg_\C,\ft_\C)} \Z\alpha$, where $\Phi(\fg_\C,\ft_\C)$ denotes the root system associated to the Cartan subalgebra $\ft_\C$ of $\fg_\C$.
Therefore, the eigenvalue spectrum of the standard bi-invariant metric on $G/Z(G)$ is given by 
\begin{equation}\label{eq:eigenvaluespectrumG/T}
\big\{ \Cas(\pi_\Lambda): \Lambda\in\mathcal R\text{ is dominant} \big\},
\end{equation}
where $\pi_\Lambda$ denotes the corresponding irreducible representation with highest weight $\Lambda$. 

Yamaguchi~\cite[Thm.~6]{Yamaguchi79} proved that the spectrum of the full flag manifold $G/T$ endowed with standard (or Killing metric) is given by 
\begin{equation}
\Big\{\!\! \Big\{ 
\underbrace{\Cas(\pi),\dots,\Cas(\pi)}_{\dim V_\pi \dim V_\pi(0)}
%\Cas(\pi)\text{ with mult.\ } \dim V_\pi \dim V_\pi(0)
:\pi\in\widehat G
\Big\}\!\!\Big\},
\end{equation}
where $V_\pi(0)$ is the weight space associated to the weight zero. 
It is important to note that $\pi\in\widehat G$ contributes $\Cas(\pi)$ as an eigenvalue of the Laplacian if and only if $\dim V_\pi(0)>0$. 
By a result of Freudenthal (see \cite[Thm.~4]{Yamaguchi79}),  $\dim V_\pi(0)>0$ if and only if the highest weight of $\pi$ lives in the root lattice $\mathcal R$.
Therefore, the eigenvalue spectrum of $G/T$ is precisely \eqref{eq:eigenvaluespectrumG/T}.

The last assertion follows easily from $\dim G/Z(G)=\dim G$ since $Z(G)$ is discrete in $G$ and $\dim G/T= \dim G-\dim T$ with $\dim T=\op{rank} G>0$.  
\end{proof}

\begin{remark}
The choice $G=\SU(2)$ in Theorem~\ref{thm:eigenvalueequivalent-diffdimensions} provides the eigenvalue equivalent pair $(P^3(\R),g_{\text{round}})$ and $(S^2,\tfrac{1}{4} g_{\text{round}})$ shown at the beginning of this subsection. 
\end{remark}

\begin{remark}\label{rem:eigenvaluespectrumdoesnothearsingularities}
%It is worth mentioning that the closed Riemannian manifolds in Theorem~\ref{thm:eigenvalueequivalent-diffdimensions} are homogeneous since $G$ acts transitively and by isometries on both.  
%Moreover, 
For any finite subgroup $\Gamma$ of the group $G$ as in Theorem~\ref{thm:eigenvalueequivalent-diffdimensions}, one can see that
\begin{equation}
\begin{aligned}
\Spec(\Gamma \backslash G/Z(G)) &=
\Big\{\!\! \Big\{ 
\underbrace{\Cas(\pi),\dots,\Cas(\pi)}_{\dim V_\pi^\Gamma \dim V_\pi}
:\pi\in\widehat K
\Big\}\!\!\Big\},
\\
\Spec(\Gamma \backslash G/T)&=
\Big\{\!\! \Big\{ 
\underbrace{\Cas(\pi),\dots,\Cas(\pi)}_{\dim V_\pi^\Gamma \dim V_\pi(0)}
:\pi\in\widehat G
\Big\}\!\!\Big\},
\end{aligned}
\end{equation}
so that in particular, they are eigenvalue equivalent. 

For $G$ simple of type $\textup{E}_8$, $\textup{F}_4$ and $\textup{G}_2$, one has that $Z(G)$ is trivial, thus any finite subgroup $\Gamma$ whose action on $G/T$ is not free provides new examples of a manifold $\Gamma\backslash G$ eigenvalue equivalent to an orbifold $\Gamma\backslash G/T$ with non-trivial singularities. 
\end{remark}

\section{Isospectral pairs of largest volume} \label{sec:highestvolume}

In this section we discuss the problem introduced in Subsection~\ref{subsec:highestvolume}: determining the isospectral pair of spherical orbifolds/space forms of largest volume. 

Recall our assumption that every spherical orbifold is endowed with the Riemannian metric with constant sectional curvature one, and that there is an equality
\begin{equation}
\vol(S^d/\Gamma) = \frac{\vol(S^d)}{|\Gamma|} 
\end{equation} 
for any finite subgroup $\Gamma$ of $\Ot(d+1)$.

\subsection{Isospectral spherical orbifold of largest volume}
The orbifold setting is simpler than that of spherical space forms.  

\begin{proof}[Proof of Theorem~\ref{thm1:highestvolume-orbifold}]
Let $\Gamma_1$ and $\Gamma_2$ be the finite subgroups of $\SO(6)$ with diagonal matrices whose entries are 
\begin{equation}
\Gamma_1:
\begin{cases}
(1,1,1,1,1,1),\\
(-1,-1,-1,-1,-1,-1),\\
(-1,-1,1,1,1,1),\\
(-1, 1,-1, 1, 1, 1), \\
(1,-1,-1, 1, 1, 1),\\
(-1, 1, 1,-1,-1,-1), \\
(1,-1, 1,-1,-1,-1), \\
(1, 1,-1,-1,-1,-1),
\end{cases}
\qquad
\Gamma_2:
\begin{cases}
(1, 1, 1, 1, 1, 1), \\
(-1,-1,-1,-1,-1,-1),\\
(-1,-1, 1, 1, 1, 1), \\
(1, 1,-1,-1, 1, 1), \\
(1, 1, 1, 1,-1,-1),\\
(-1,-1,-1,-1, 1, 1), \\
(-1,-1, 1, 1,-1,-1), \\
(1, 1,-1,-1,-1,-1).
\end{cases}
\end{equation}

\cite[Ex.~2.4]{RossettiSchuethWeilandt08} shows that $\Gamma_1$ and $\Gamma_2$ are almost conjugate (but not conjugate) in $\SO(6)$. 
Consequently, $S^5/\Gamma_1$ and $S^5/\Gamma_2$ are (strongly) isospectral by Theorem~\ref{thm:Sunadamethod}.
That $S^5/\Gamma_1$ and $S^5/\Gamma_2$ are not isometric is proven in \cite[Ex.~2.9]{RossettiSchuethWeilandt08}. 

For any $d\geq 5$, let us denote by $\widetilde \Gamma_i$ the subgroup of $\SO(d+1)$ given by adding $d-5$ entries equal to $1$ to each element in $\Gamma_i$. 
It is not difficult to see that $S^{d}/\widetilde \Gamma_1$ and $S^{d}/\widetilde \Gamma_2$ are again isospectral and non-isometric because $\widetilde \Gamma_1$ and $\widetilde \Gamma_2$ are almost conjugate in $\SO(d+1)$. 

Since $\widetilde \Gamma_1$ and $\widetilde \Gamma_2$ each have $8$ elements, we have that $\vol(S^{d+1}/\widetilde \Gamma_i)=\frac{\vol(S^{d+1})}{8}$ for $i=1,2$, and the assertion follows. 
\end{proof}

It is not known whether the above example attains the largest volume asked in Question~\ref{question:highestvolume-orbifold}. 
The following result provides a lower bound for the mentioned highest volume. 

\begin{proposition}
There are no pairs of $d$-dimensional isospectral and non-isometric spherical orbifolds with volume greater than or equal to $\frac{\vol(S^d)}{3}$.
\end{proposition}

\begin{proof}
We have to prove that two isospectral spherical orbifolds of volume $\frac{1}{q}\vol(S^{d})$ are isometric for $q=1,2,3$.
The case $q=1$ is trivial.

Let $\Gamma$ be a subgroup of $\Ot(d+1)$ of order $2$ and let $\gamma$ denote the non-trivial element in $\Gamma$. 
Since $\gamma^2=\op{I}_{d+1}$, the eigenvalues of $\gamma$ are $\pm1$, say $-1$ with multiplicity $m$ and $+1$ with multiplicity $(d+1-m)$. 
By \eqref{eq:F_Gamma-expression} the corresponding spectral generating function is given by 
\begin{equation*}
\begin{aligned}
F_{\Gamma}(z) 
&= \frac{1-z^2}{2}\left( \frac{1}{(z-1)^{d+1}} + \frac{1}{(z-1)^{d+1-m}(z+1)^{m}} \right)
\\
&= -\frac{z+1}{2(z-1)^{d}} - \frac{1}{2(z-1)^{d-m}(z+1)^{m-1}}. 
\end{aligned}
\end{equation*}
It follows immediately that $m-1$ is the order of the pole of $F_{\Gamma}(z)$ at $z=-1$, and consequently $m$ is determined by $\Spec(S^{d+1}/\Gamma)$. 
Since the number of times that the eigenvalue $-1$ is in the spectrum of $\gamma$ determines the conjugacy class of $\Gamma$ in $\Ot(d+1)$, we conclude that any two isospectral spherical orbifolds of volume $\frac12 \vol(S^{d})$ are necessarily isometric.

Let $\Gamma$ be a subgroup of $\Ot(d+1)$ of order $3$. 
Let $\gamma$ denote any non-trivial element of $\Gamma$ so that $\Gamma=\{\op{I}_{d+1},\gamma,\gamma^2\}$. 
Since $\gamma$ and $\gamma^2$ have order $3$, their eigenvalues are as follows: 
$m$-times $\xi_3=e^{\frac{2\pi\mi}{3}}$, $m$-times $\xi_3^2$, and $(d+1-2m)$-times $1$, for some integer $m$ satisfying $1\leq m\leq \frac{d+1}{2}$. 
Hence, 
\begin{equation*}
\begin{aligned}
F_{\Gamma}(z) 
	&= \frac{1-z^2}{3}\left( \frac{1}{(z-1)^{d+1}} + \frac{2}{(z-1)^{d+1-2m} (z-\xi_3)^{m}(z-\xi_3^2)^{m}} \right)
\\
	&= -\frac{z+1}{3(z-1)^{d}} - \frac{1-z^2}{3(z-1)^{d+1-m}(z-\xi_3)^{m}(z-\xi_3^2)^{m}}. 
\end{aligned}
\end{equation*}
As above, $m$ is an spectral invariant because is the order of the pole of $F_{\Gamma}(z)$ at $z=\xi_3$. 
Furthermore, since $m$ determines the conjugacy class of $\Gamma$ in $\Ot(d+1)$, we conclude that any two isospectral spherical orbifolds of volume $\frac13 \vol(S^{d+1})$ are necessarily isometric. 
\end{proof}

%Before moving to the manifold case, some remarks are in order. 

\begin{remark}\label{rem:orbifold-dim<=4}
%Theorem~\ref{thm1:highestvolume-orbifold} is currently optimal in terms of the dimension since it is not known whether there exists a pair of isospectral non-isometric spherical orbifolds of dimension at most $4$.
%Actually, there are good reasons to think that such example does not exist. 
It is not known whether there exist isospectral and non-isometric spherical orbifolds of dimension $\leq 4$. 
There are two good reasons to believe that no such orbifolds exist. 
On the one hand, Bari and Hunsicker~\cite{ShamsHunsicker17} proved that there are no isospectral and non-isometric lens orbifolds of dimension $\leq4$. 
On the other hand, Vásquez~\cite[Prop.~2.4]{Vasquez18} proved that there are no almost conjugate, non-conjugate subgroups of the double cover $\Spin(4)\simeq \SU(2)\times\SU(2)$ of $\SO(4)$. 
\end{remark}

\subsection{Isospectral spherical space forms of largest volume}
We now restrict our attention to manifolds covered by round spheres. 
This case is much more complicated than the previous one. 
In particular, we need several preliminaries.

Question~\ref{question:highestvolume-manifold} is not applicable in even dimensions because in these dimensions it is known that no examples of isospectral, non-isometric spherical space forms exist. 

	In fact, the only spherical space forms of dimension $2n$ are $S^{2n}$ and $P^{2n}(\R)$, and they are not isospectral because there are no round metrics on them with the same volume and total scalar curvature. 

Furthermore, Ikeda~\cite{Ikeda80_3-dimI} proved that any two isospectral $3$-dimensional spherical space forms are necessarily isometric. We will therefore restrict our attention on spherical space forms of odd dimension $\geq5$. 

The following spectral invariants, proven by Ikeda~\cite[Cor.~2.4, 2.8]{Ikeda80_3-dimI}, will be very useful. 

\begin{proposition}[Ikeda]\label{prop:Ikeda-sameorder}
Let $S^{2n-1}/\Gamma_1$ and $S^{2n-1}/\Gamma_2$ be isospectral spherical space forms. 
Then, $|\Gamma_1|=|\Gamma_2|$ and $\Gamma_1,\Gamma_2$ have the same set of orders of their elements.
In particular, $\Gamma_1$ is cyclic if and only if $\Gamma_2$ is cyclic.
\end{proposition}

The first goal is to classify non-cyclic finite subgroups of low order acting freely on an odd-dimensional sphere. 
This will allow us to show that the isospectral pairs of spherical space forms of largest volume are realized by lens spaces, at least in low dimensions.

\begin{lemma}\label{lem:non-cyclic24}
If the fundamental group $\Gamma$ of a spherical space form $S^{2n-1}/\Gamma$ is non-cyclic and has order strictly less than $24$, then $\Gamma$ is isomorphic to one of the following groups:
\begin{itemize}
\item the quaternion group $Q_8:=\langle B,R\mid B^4=e,\, R^2=B^2,\, RBR^{-1}=B^3\rangle$ of order $8$,

\item the group $P_{12}:=\langle A,B\mid A^3=B^4=e,\, BAB^{-1}=A^2\rangle$ of order $12$,

\item the generalized quaternion group $Q_{16}:=\langle B,R\mid B^8=e,\, R^2=B^4,\, RBR^{-1}=B^7\rangle$ of order $16$,

\item the group $P_{20}:=\langle A,B\mid A^5=B^4=e,\, BAB^{-1}=A^4\rangle$ of order $20$. 
\end{itemize}
Moreover, for each group $H$ in the above list and $m\in\N$, there is up to isometry exactly one spherical space form with fundamental group isomorphic to $H$ and dimension $4m+3$.
\end{lemma}

\begin{proof}
The proof is deeply based on the classification of spherical space forms obtained by Wolf in \cite{Wolf-book}. 
We will use Wolf's notation to facilitate the reading and refer to the article \cite{Wolf01} which has a summary of it. 

The strategy of the classification of spherical space forms is to classify abstract finite groups $H$ admitting fixed point free real representations (i.e.\ $\rho:H\to\SO(d)$ such that $S^{d-1}/\rho(H)$ is a spherical space form), called \emph{fixed point free groups}, and then to classify, for each fixed point free group $H$, all of its fixed point free real representations. 
We first show that all non-cyclic fixed point free groups of order less than $24$ are those listed in the statement. 

Fixed point free groups divide into six types: Type I,II,\dots, VI (see \cite[\S3]{Wolf01}). 
We will show that the number of non-cyclic fixed point free groups with at most $23$ elements is $2$ for Types I and II, and $0$ for the rest of types. 

The fixed point free groups of Type I are of the form (see \cite[Prop.~3.1]{Wolf01})
\begin{equation}
H_d(m,n,r):=\langle A,B\mid A^m=B^n=e,\; BAB^{-1}=A^r\rangle
\end{equation}
for some $m,n,r,d\in\N$ satisfying that 
\begin{equation*}
\begin{cases}
\textup{(I.1)} & \gcd(n(r-1),m)=1, \\
\textup{(I.2)} & r^n\equiv 1\pmod m,\; 1\leq r\leq m,\\ 
\textup{(I.3)} & \text{$d$ is the order of $r$ in $\Z_m^\times$}, \\
\textup{(I.4)} & \text{$n/d$ is divisible by every prime divisor of $d$.}
\end{cases}
\end{equation*}
Note $|H_{d}(m,n,r)|=mn$. 
One can easily see that 
\begin{equation}\label{eq:cyclic}
H_{d}(m,n,r)\text{ is cyclic} \iff m=1\iff d=1\iff r=1. 
\end{equation}
Furthermore, $m\neq 2$ since otherwise $r=1$ by (I.2) and \eqref{eq:cyclic} gives a contradiction.
Also, $n= 1$ gives $d=1$ by \textup{(I.4)}, so the group is trivial. 

\begin{claim}\label{claim:typeI}
The only non-cyclic fixed point free groups of Type I of order $<24$ are $H_{2}(3,4,2)\simeq P_{12}$ and $H_2(5,4,4)\simeq P_{20}$. 
\end{claim}

\begin{proof}
\renewcommand{\qedsymbol}{$\blacksquare$}
Suppose $H_{d}(m,n,r)$ is non-cyclic and $mn<24$.
We have $d\geq2$, $m\geq 3$ and, $n\geq4$ by \eqref{eq:cyclic}. 

It follows from \textup{(I.4)} that $d=2$ and $n=4$.
In fact, 
$d\geq 4$ implies $n\geq 2d=8$, thus $mn\geq 24$; 
$d=3$ leads to $n=9$ thus $mn\geq 27$;
$d=2$ forces $n= 2n'$ with $n'$ divisible by $2$, thus $24>mn\geq 6n'$, so $n'=2$. 

Now, \textup{(I.1)} implies $m$ is odd, thus $m=3$ or $m=5$. 
If $m=3$ (resp.\ $m=5$), then $r=2$ (resp.\ $r=4$) by \textup{(I.3)}, and the proof is complete since \textup{(I.1)} and \textup{(I.2)} hold. 
\end{proof}

A group of Type II is of the form (see \cite[Prop.~3.1]{Wolf01})
\begin{equation}
H_d(m,n,r,s,t):=
\left\langle A,B,R\mid  
\begin{array}{l}
	A^m=B^n=e,\; BAB^{-1}=A^r,\\ R^2=B^{n/2},\; RAR^{-1}=A^s, \; RBR^{-1}=B^t 
\end{array}
\right\rangle
\end{equation}
for some $m,n,r,d,s,t\in\N$ satisfying \textup{(I.1)--(I.4)} and also 
\begin{equation}
\begin{cases}
\textup{(II.1)} & s^2\equiv 1\pmod m,\; 1\leq s\leq m\\
\textup{(II.2)} & r^{t-1}\equiv 1\pmod m,\\
\textup{(II.3)} & \textup{$n=2^u n''$ with $u\geq2$ and $\gcd(n'',2)=1$},\\
\textup{(II.4)} & t\equiv -1\pmod {2^u},\; t^2\equiv 1\pmod n,\; 1\leq t\leq n
\end{cases}
\end{equation}
Note that $\langle A,B\rangle\simeq H_d(m,n,r)$ is a proper subgroup of $H_d(m,n,r,s,t)$ and $|H_d(m,n,r,s,t)|= 2mn$. 

\begin{claim}
The only non-cyclic fixed point free groups of Type II of order less than $24$ are $H_1(1,4,1,1,3)\simeq Q_{8}$ and $H_1(1,8,1,1,7)\simeq Q_{16}$. 
\end{claim}

\begin{proof}
\renewcommand{\qedsymbol}{$\blacksquare$}
It is clear that the subgroup $\langle A,B\rangle$ must be cyclic since otherwise $mn\geq 12$ by Claim~\ref{claim:typeI}.
Hence, $d=m=r=1$ by \eqref{eq:cyclic}, $s=1$ by \textup{(II.1)}. 
Note \textup{(II.2)} hold trivially. 

Since $|H_d(m,n,r,s,t)|=2n<24$, \textup{(II.3)} forces $n=4$ or $n=8$, which give $t=3$ and $t=7$ respectively, by \textup{(II.4)}. 
The first case give $H_1(1,4,1,1,3) = \langle B,R \mid  B^4=e,\; R^2=B^{2},\; RBR^{-1}=B^3\rangle$, which is isomorphic to the quaternion group $Q_8=\{\pm1,\pm\mi,\pm \mj,\pm \mk\}$ via $B\mapsto \mi$ and $R\mapsto \mj$.
The second case gives $H_1(1,8,1,1,7) = \langle B,R \mid  B^8=e,\; R^2=B^{4},\; RBR^{-1}=B^7\rangle$, which is precisely the usual presentation of the quaternion group $Q_{16}$. 
\end{proof}

Any group of Type III (resp.\ Type IV) has a proper subgroup isomorphic to $H_{d}(m,n,r)$ of coindex $8$ (resp.\ $16$) with $n$ odd and divisible by $3$ (see \cite[Prop.~3.1]{Wolf01}).
Thus, it has $8mn\geq 24$ (resp.\ $16mn\geq 48$) elements. 
Furthermore, the groups of Type V and VI have at least 120 elements since $\SL_2(\F_5)$ is a subgroup of them (\cite[Top of page 327]{Wolf01}). 
This concludes the first part of the proof.

The ways that a fixed point free group $H$ embeds into $\SO(2n)$ via a fixed point free real representation $\rho$ to give all the isometry classes of $(2n-1)$-dimensional spherical space forms of the form $S^{2n-1}/\rho(H)$ is quite complicated to be explained. 
See \cite[\S4--5]{Wolf01} for a nice summary, or \cite[\S7.2--3]{Wolf-book} for their classifications. 
However, our claim follows immediately from \cite[Prop.~5.3]{Wolf01}.
More precisely, for any of our four choices of $H$, any two irreducible complex fixed point free representations of $H$ are related by a composition with an automorphism of $H$ (i.e.\ $\pi_{k,l}\cong \pi_{1,1}\circ \mathtt{A}$ or $\alpha_{k,l}\cong\alpha_{1,1}\circ \mathtt{A}$ for some $\mathtt{A}\in\op{Aut}(H)$ for all $k,l$ admitted, when $H$ is of Type I or II respectively), and have degree $\delta(H)=2$.
Hence, any spherical space form with fundamental group isomorphic to $H$ is isometric to 
\begin{equation}
S^{4m-1}/\big(\underbrace{(\pi\oplus\bar\pi)\oplus \dots \oplus (\pi\oplus\bar\pi) }_{m\text{-times}}\big)(H)
\end{equation}
for some $m\in\N$, where $\pi$ denotes any irreducible complex fixed point free representations of $H$. 
\end{proof}

Lemma~\ref{lem:non-cyclic24} and Proposition~\ref{prop:Ikeda-sameorder} imply that any pair of odd-dimensional isospectral and non-isometric spherical space forms $S^{2n-1}/\Gamma_1,S^{2n-1}/\Gamma_2$ with $|\Gamma_i|<24$ are necessarily lens spaces. Consequently,  $\Gamma_1$ and $\Gamma_2$ are cyclic. 
Therefore, if in a fixed odd dimension $2n-1$ there are isospectral lens spaces with fundamental groups of order strictly less than $24$, then the isospectral pairs of $(2n-1)$-dimensional spherical space forms with largest volume will be lens spaces. 
For instance, this is the case for the isospectral pair 
\begin{equation}\label{minimalIkeda'spair}
\{L(11;1,2,3),\, L(11;1,2,4)\}
\end{equation}
of $5$-dimensional lens spaces found by Ikeda~\cite{Ikeda80_3-dimI}. 
See Table~\ref{table:lens} for many more examples found with the computer (see e.g.\ \cite{Lauret-computationalstudy}). 
The next goal is to construct higher dimensional isospectral pairs of lens spaces from such pairs in low dimension. 
To facilitate the next statement, we introduce some notation.

\begin{notation}
We think of the parameters $s=(s_1,\dots,s_n)$ of a $(2n-1)$-dimensional lens space $L(q;s)$ with fundamental group of order $q$ as a list $[s_1,\dots,s_n]$. 
As usual, the concatenation of lists is given by $[a_1,\dots,a_m] +[b_1,\dots,b_n] =[a_1,\dots,a_m,b_1,\dots,b_n]$ and $r\cdot [a_1,\dots,a_n] =[a_1,\dots,a_n]+\dots+[a_1,\dots,a_n]$ for $r\in\N$. 
For $q\in\N$, we set $q_0=\varphi(q)/2$, where $\varphi(q)=|\Z_q^\times|$ is Euler's totient function, defined to be the number of units in $\Z_q=\Z/q\Z$. 
Let $t(q)=[t_1,\dots,t_{q_0}]$, where $\{\pm t_1,\dots,\pm t_{q_0}\}$ is a representative set of $\Z_q^\times$, i.e.\ for any $m\in\Z$ prime to $q$ one has $m\equiv t_i\pmod q$ or $m\equiv -t_i\pmod q$ for some (unique) $i$. 
For instance, $t(7)=[1,2,3]$ and $t(12)=[1,5]$. 
\end{notation}

\begin{theorem}\label{thm:increasingdimension}
If the $(2n-1)$-dimensional lens spaces $L(q;s)$ and $L(q;s')$ are isospectral and non-isometric, then the $(2n-1+r\varphi(q))$-dimensional lens spaces 
\begin{equation}
L(q; s+r\cdot t(q))
\quad\text{ and }\quad 
L(q;s'+r\cdot t(q))
\end{equation}
are isospectral and non-isometric for every $r\in\N$. 
\end{theorem}

\begin{proof}
We fix $r\in\N$ and write $L=L(q; s+r\cdot t(q))$ and $L'=L(q;s'+r\cdot t(q))$. 
It is a simple matter to check that $L$ and $L'$ are non-isometric using Proposition~\ref{prop:lens-isometries}.
We have to show that $F_L(z)=F_{L'}(z)$. 
Let us denote by $\Phi_q(z)$ the $q$th cyclotomic polynomial; 
notice $\Phi_q(z)=\prod_{j=1}^{q_0} (z-\xi_q^{t_i})(z-\xi_q^{-t_i})$. 
From \ref{eq:F_L-expression}, it follows that 
\begin{equation}\label{eq:F_L(q;s+r.t)}
\begin{aligned}
F_{L}(z)&
= \frac{1-z^2}{q} \sum_{h=0}^{q-1} \frac{1}{ \prod_{i=1}^n (z-\xi_q^{hs_i})(z-\xi_q^{-hs_i})} \left(\frac{1}{\prod_{j=1}^{q_0} (z-\xi_q^{ht_j}) (z-\xi_q^{-ht_j})}\right)^r 
\\ &
= \frac{1-z^2}{q} \frac{1}{(z-1)^{2n+2rq_0}}
%\\ & \quad 
+\frac{1}{\Phi_q(z)^{r}}  \frac{1-z^2}{q} \sum_{h=1}^{q-1} \frac{1}{ \prod_{i=1}^n (z-\xi_q^{hs_i})(z-\xi_q^{-hs_i})} 
\\ &
= \frac{1-z^2}{q} \frac{1}{(z-1)^{2n+2rq_0}}
+\frac{1}{\Phi_q(z)^{r}}  \left(F_{L(q;s)}(z)- \frac{1-z^2}{q} \frac{1}{(z-1)^{2n}} \right)
.
\end{aligned}
\end{equation}
Now, using that $F_{L(q;s)}(z)=F_{L(q;s')}(z)$ because $L(q;s)$ and $L(q;s')$ are isospectral by hypothesis, and returning over the steps in \eqref{eq:F_L(q;s+r.t)} for $s'$, we obtain $F_L(z)=F_{L'}(z)$ as required. 
\end{proof}

\begin{table}
\caption{Isospectral lens spaces for low values of $q$ and $n$.}
\label{table:lens}
$
\begin{array}{cccl}
q & q_0 & n & \text{Parameters $s$ of isospectral lens spaces} \\ \hline\hline
11 & 5 & 3 & [1, 2, 3], [1, 2, 4]
%	[1, 2, 5], [1, 3, 4]
\\
&& 7 & [1, 1, 2, 2, 3, 3, 4], [1, 1, 2, 2, 3, 3, 5]%, [1, 1, 2, 2, 3, 4, 4]
 \\
&& 11 & [1, 1, 1, 2, 2, 2, 3, 3, 3, 5, 5], [1, 1, 1, 2, 2, 2, 3, 3, 4, 4, 4]
\\ \hline
13 & 6 & 3 & [1, 2, 3], [1, 2, 4]
\\
&& 4 & [1, 2, 3, 4], [1, 2, 3, 5]%, [1, 2, 3, 6]
\\
&& 7 & [1, 1, 2, 2, 3, 3, 5], [1, 1, 2, 2, 3, 4, 4]
\\ 
&& 8 & 
	[1, 1, 2, 2, 3, 3, 4, 4], [1, 1, 2, 2, 3, 3, 5, 5]
%	[1, 1, 2, 2, 3, 3, 4, 5], [1, 1, 2, 2, 3, 3, 4, 6]
%	[1, 1, 2, 2, 3, 5, 5, 6], [1, 1, 2, 3, 3, 4, 4, 5]
%	[1, 1, 2, 2, 3, 3, 5, 6], [1, 1, 2, 2, 3, 4, 6, 6]
\\ \hline
16 & 4 & 5 & [1, 1, 3, 3, 5], [1, 1, 3, 3, 7]
\\ \hline
17 & 8 & 3 & 
	[1, 2, 5], [1, 2, 6]]
%	[[1, 2, 7], [1, 3, 4]]
\\
&& 4 & 
	[1, 2, 3, 5], [1, 2, 3, 6]%, [1, 2, 3, 8]]
%	[[1, 2, 3, 7], [1, 2, 4, 5]]
%	[[1, 2, 6, 7], [1, 3, 4, 5]]
\\
&& 5 & 
	[1, 2, 3, 4, 5], [1, 2, 3, 4, 6]%, [1, 2, 3, 4, 7], [1, 2, 3, 4, 8], [1, 2, 3, 5, 6]]
	%[[1, 2, 3, 5, 8], [1, 2, 3, 6, 7]
\\
&& 6 & 
	[1, 2, 3, 4, 5, 6], [1, 2, 3, 4, 5, 7]%, [1, 2, 3, 4, 5, 8], [1, 2, 3, 4, 6, 8]]
\\
&& 10 & 
	[1, 1, 2, 2, 3, 4, 5, 7, 7, 8], [1, 1, 2, 2, 3, 4, 6, 7, 7, 8]%, [1, 1, 2, 3, 3, 4, 4, 5, 6, 7]]
%	\\ &&
%	[[1, 1, 2, 2, 3, 3, 4, 5, 8, 8], [1, 1, 2, 2, 3, 3, 4, 6, 6, 7]]
%	\\ &&
%	[[1, 1, 2, 2, 3, 3, 4, 5, 7, 7], [1, 1, 2, 2, 3, 3, 5, 6, 6, 8]]
%	\\ &&
%	[[1, 1, 2, 2, 3, 5, 6, 6, 7, 8], [1, 1, 2, 3, 3, 4, 5, 5, 6, 8]]
%	\\ &&
%	[[1, 1, 2, 2, 3, 4, 4, 5, 7, 8], [1, 1, 2, 2, 3, 4, 4, 6, 7, 8]]
\\ \hline
19 & 9 & 3 & 
	[1, 2, 7], [1, 3, 4]
\\
&& 4 & [1, 2, 6, 8], [1, 3, 4, 5]
\\
&&5&
	[1, 2, 3, 4, 5], [1, 2, 3, 4, 6]
%	[[1, 2, 3, 5, 6], [1, 2, 3, 5, 9], [1, 2, 3, 6, 8], [1, 2, 3, 7, 9]]
%	[[1, 2, 3, 4, 7], [1, 2, 3, 4, 9], [1, 2, 3, 5, 7], [1, 2, 3, 6, 7]]
%	[[1, 2, 3, 7, 8], [1, 2, 5, 6, 8]]
\\
&&6&
	[1, 2, 3, 4, 5, 6], [1, 2, 3, 4, 5, 7]%, [1, 2, 3, 4, 5, 8], [1, 2, 3, 4, 5, 9], [1, 2, 3, 4, 6, 7], [1, 2, 3, 4, 6, 8], [1, 2, 3, 4, 6, 9]]
	%[[1, 2, 3, 4, 7, 9], [1, 2, 3, 5, 6, 9], [1, 2, 3, 5, 7, 8]]
\\
&& 7 & [1, 2, 3, 4, 5, 6, 7], [1, 2, 3, 4, 5, 6, 8]%, [1, 2, 3, 4, 5, 6, 9], [1, 2, 3, 4, 5, 8, 9]]
\\
&& 11 & 
	[1, 1, 2, 2, 3, 3, 4, 6, 6, 7, 7], [1, 1, 2, 2, 3, 3, 4, 6, 6, 9, 9]
	%[[1, 1, 2, 3, 3, 4, 4, 5, 6, 8, 9], [1, 1, 2, 3, 3, 4, 4, 6, 7, 8, 9]]
\\  \hline
20 & 4 & 5 & [1, 1, 3, 3, 7], [1, 1, 3, 3, 9]
\\ \hline
21 & 6 & 4 & [1, 2, 4, 5], [1, 2, 4, 8]
\\
&& 9 & [1, 1, 2, 2, 4, 4, 5, 5, 10], [1, 1, 2, 2, 4, 4, 5, 8, 8]
\\
&& 14 & [1, 1, 1, 2, 2, 2, 4, 4, 4, 5, 5, 5, 10, 10], [1, 1, 1, 2, 2, 2, 4, 4, 4, 5, 5, 8, 8, 8]
\\ \hline
22 & 5 & 3 & [1, 3, 5], [1, 3, 7]
\\
&&7 & [1, 1, 3, 3, 5, 5, 9], [1, 1, 3, 3, 5, 7, 7]%, [1, 1, 3, 3, 5, 9, 9]]
\\
&& 11 & [1, 1, 1, 3, 3, 3, 5, 5, 5, 9, 9], [1, 1, 1, 3, 3, 3, 5, 5, 9, 9, 9]
\\ \hline
23 & 11 & 4 & 
	[1, 2, 4, 5], [1, 2, 4, 8]
%	[[1, 2, 5, 7], [1, 2, 5, 9]]
%	[[1, 2, 6, 10], [1, 3, 4, 9]]
\\
&& 5 & 
	[1, 2, 3, 4, 11], [1, 2, 3, 5, 6]
%	[[1, 2, 3, 4, 7], [1, 2, 3, 4, 8]]
%	[[1, 2, 3, 6, 7], [1, 2, 3, 6, 10], [1, 2, 3, 7, 9]]
%	[[1, 2, 3, 5, 9], [1, 2, 3, 6, 8]]
%	[[1, 2, 5, 6, 8], [1, 2, 5, 6, 10]]
%	[[1, 2, 3, 8, 10], [1, 2, 3, 9, 10], [1, 2, 4, 9, 10], [1, 2, 6, 8, 9]]
%	[[1, 2, 6, 7, 10], [1, 3, 4, 5, 7]]
%	[[1, 2, 4, 5, 11], [1, 2, 4, 7, 11]]
\\
&& 6 & 
[1, 2, 3, 4, 5, 10], [1, 2, 3, 4, 6, 7]%, [1, 2, 3, 4, 6, 10], [1, 2, 3, 4, 6, 11], [1, 2, 3, 4, 7, 8], [1, 2, 3, 4, 8, 11]]
%[[1, 2, 3, 4, 7, 11], [1, 2, 3, 4, 9, 11], [1, 2, 3, 5, 6, 7], [1, 2, 3, 5, 6, 8], [1, 2, 3, 5, 6, 10], [1, 2, 3, 5, 7, 8], [1, 2, 3, 5, 7, 10]]
%[[1, 2, 3, 4, 5, 8], [1, 2, 3, 4, 5, 9], [1, 2, 3, 4, 6, 8]]
%[[1, 2, 3, 4, 5, 11], [1, 2, 3, 4, 6, 9]]
%[[1, 2, 3, 4, 7, 9], [1, 2, 3, 5, 6, 9]]
%[[1, 2, 3, 4, 9, 10], [1, 2, 3, 5, 9, 10], [1, 2, 3, 5, 9, 11], [1, 2, 3, 6, 7, 9], [1, 2, 3, 8, 9, 11]]
%[[1, 2, 3, 5, 7, 11], [1, 2, 3, 5, 8, 9], [1, 2, 3, 5, 8, 11], [1, 2, 3, 6, 9, 10]]
%[[1, 2, 3, 4, 8, 10], [1, 2, 3, 6, 8, 11]]
%[[1, 2, 3, 8, 9, 10], [1, 2, 6, 8, 9, 10]]
\\
&& 7 &
[1, 2, 3, 4, 5, 6, 7], [1, 2, 3, 4, 5, 6, 8]%, [1, 2, 3, 4, 5, 6, 9], [1, 2, 3, 4, 5, 6, 11], [1, 2, 3, 4, 5, 7, 9]]
%[[1, 2, 3, 4, 5, 6, 10], [1, 2, 3, 4, 5, 7, 8], [1, 2, 3, 4, 5, 8, 10], [1, 2, 3, 4, 6, 7, 8], [1, 2, 3, 4, 6, 8, 11]]
%[[1, 2, 3, 4, 5, 7, 11], [1, 2, 3, 4, 5, 9, 11], [1, 2, 3, 4, 6, 7, 9], [1, 2, 3, 4, 6, 7, 11], [1, 2, 3, 4, 6, 9, 10], [1, 2, 3, 4, 7, 8, 10], [1, 2, 3, 4, 7, 9, 10], [1, 2, 3, 4, 7, 9, 11], [1, 2, 3, 4, 8, 9, 11], [1, 2, 3, 5, 6, 7, 8]]
%[[1, 2, 3, 4, 5, 7, 10], [1, 2, 3, 4, 5, 8, 9], [1, 2, 3, 4, 5, 8, 11], [1, 2, 3, 4, 5, 10, 11], [1, 2, 3, 4, 6, 8, 9], [1, 2, 3, 4, 6, 9, 11]]
%[[1, 2, 3, 4, 8, 9, 10], [1, 2, 3, 5, 7, 8, 11], [1, 2, 3, 5, 8, 9, 11]]
\\
&& 8 &
[1, 2, 3, 4, 5, 6, 7, 8], [1, 2, 3, 4, 5, 6, 7, 9]%, [1, 2, 3, 4, 5, 6, 7, 10], [1, 2, 3, 4, 5, 6, 7, 11], [1, 2, 3, 4, 5, 6, 8, 9], [1, 2, 3, 4, 5, 6, 8, 10], [1, 2, 3, 4, 5, 6, 8, 11], [1, 2, 3, 4, 5, 6, 9, 10], [1, 2, 3, 4, 5, 6, 9, 11], [1, 2, 3, 4, 5, 6, 10, 11], [1, 2, 3, 4, 5, 7, 8, 9], [1, 2, 3, 4, 5, 7, 9, 10]]
%[[1, 2, 3, 4, 5, 7, 10, 11], [1, 2, 3, 4, 5, 8, 9, 11], [1, 2, 3, 4, 6, 7, 9, 11]]
\\
&& 9 & 
[1, 2, 3, 4, 5, 6, 7, 8, 9], [1, 2, 3, 4, 5, 6, 7, 8, 10] %, [1, 2, 3, 4, 5, 6, 7, 8, 11], [1, 2, 3, 4, 5, 6, 7, 9, 10], [1, 2, 3, 4, 5, 6, 8, 10, 11]]
\\
&& 13 & [1, 1, 2, 3, 3, 4, 4, 5, 6, 7, 8, 9, 10], [1, 1, 2, 3, 3, 4, 4, 5, 6, 7, 8, 10, 11]
\\
&& 14 & [1, 1, 2, 2, 3, 5, 6, 7, 8, 8, 9, 9, 10, 11] [1, 1, 2, 2, 4, 5, 6, 7, 8, 8, 9, 9, 10, 11]
\end{array}
$
\end{table}

\begin{table}
\caption{Existence of isospectral pairs of lens spaces for low values of $q$ and $n$.}
\label{table:existencelens}
$
\begin{array}{cccccccccc}
n\ba q  & \underset{q_0=5}{11}  & \underset{q_0=6}{13} & \underset{q_0=4}{16}& \underset{q_0=8}{17}& \underset{q_0=9}{19} & \underset{q_0=4}{20} & \underset{q_0=6}{21} & \underset{q_0=5}{22}& \underset{q_0=11}{23}
\\[2mm] \hline \rule{0pt}{14pt}
 3&\otimes 	&\times	&-		&\times	&\times	&-		&-		&\times	&- \\
 4&-		&\otimes&-		&\times	&\times	&-		&\times	&-		&\times \\
 5&-		&-		&\otimes&\times	&\times	&\times	& -		&-		&\times \\
 6&-		&-		&-		&\otimes&\times	&-		&-		&-		&\times \\
 7&\otimes 	&\times	&-		&- 		&\times	&-		&-		&\times	&\times \\
 8&\otimes  &\times	&-		&- 		&- 		&-		&-		&\times	&\times \\
 9&- 		&\otimes&\times	&-		&-		&\times	&\times &-		&\times \\
10&- 		&\otimes&-		&\times	&-		&-		&\times	&-		&- \\
11&\otimes 	&-		&-		&\times	&\times	&-		&-		&\times	&- \\
12&\otimes 	&-		&-		&\times	&\times	&-		&-		&\times	&- \\
13&\otimes 	&\times	&\times	&\times	&\times	&\times	&-		&\times	&\times \\ 
14&- 		&\otimes&-		&\times	&\times	&-		&\times	&-		&\times \\ 
%15&- 		&\otimes&-		&- 		&\times	&-		&\times	&-		&\times \\
%16&\otimes	&\times	&-		&- 		&\times	&-		&\times	&\times	&\times \\
%17&\otimes 	&-		&\times	&-		&- 		&\times	&-		&\times	&\times \\
%18&\otimes	&-		&-		&\times &		&-		&-		&\times	&\times \\
%19&-		&\otimes&-		&\times &		&-		&-		&-		&\times \\
%20&-		&\otimes&-		&\times	&\times	&-		&\times	&-		&\times \\
%21&\otimes	&\times	&\times	&\times	&\times	&\times	&\times	&\times	& \\
%22&\otimes	&\times	&-		&\times	&\times	&-		&\times	&\times	& \\
%23&\otimes	&-		&- 		&		&\times	&-		&-		&\times	& \\
%24&-		&-		&- 		&		&\times	&-		&-		&-		&\times \\
%25&-		&\otimes&\times	&		&\times	&\times	&-		&-		&\times  \\
%26&\otimes	&\times	&-		&\times &		&-		&\times	&\times	&\times  \\
%27&\otimes	&\times	&-		&\times &		&-		&\times	&\times	&\times  \\
%28&\otimes	&\times	&-		&\times &		&-		&\times	&\times	&\times  \\
%29&- 		&- 		&\otimes&\times	&\times	&\times	&-		&-		&\times \\
%30&- 		&-		&-		&\otimes&\times	&-		&-		&-		&\times  \\
%31&\otimes	&\times	&-		& 		&\times	&-		&-		&\times	&\times  \\
%32&\otimes	&\times	&-		&		&\times	&-		&\times	&\times	&  \\
%33&\otimes	&\times	&\times	&		&\times	&\times	&\times	&\times	&  \\
%34&- 		&\otimes&-		&\times	&\times	&-		&\times	&-		&  \\
%35&- 		&		&-		&\times &		&-		&		&-		&\times  \\
%36&\otimes	&		&-		&\times &		&-		&		&\times	&\times  \\
%37&\otimes	&\times	&\times	&\times	&		&\times	&		&\times	&\times  \\
%38&\otimes	&\times	&-		&\times	&\times	&-		&\times	&\times	&\times  \\
%39&- 		&\otimes&-		&		&\times	&-		&\times	&-		&\times  \\
%40&- 		&\otimes& 		&		&\times	&-		&\times	&-		&\times  \\
\end{array}
$

\medskip 

\textbf{References:} 
$q_0=\frac{\varphi(q)}{2}$, 
$\times $ means that there is an isospectral pair,
$\otimes$ means that there is an isospectral pair which is known to be of largest volume, 
$-$ means there are no isospectral pairs. 

There are no isospectral pairs for $n\leq 14$ and $q\leq 10$ or $q=12,14,15,18$.
\end{table}

This result is very useful in constructing isospectral and non-isometric pairs of lens spaces in arbitrary large dimension and with fundamental group of a fixed size. 
For instance, the pair in \eqref{minimalIkeda'spair} implies the existence of $(2n+3)$-dimensional isospectral and non-isometric lens spaces of volume $\vol(S^{2n+3})/11$ for every $n\in\N$. 

We call a lens space \emph{irreducible} if it cannot be constructed from a lower dimensional lens space by adding the coefficients $t(q)=[t_1,\dots,t_{q_0}]$. 
In other words, $L(q;s)$ is irreducible if it is not isometric to $L(q;s_0+t(q))$ for any $s_0$, or equivalent, there is $m\in\Z^\times$ such that $m\not\equiv \pm s_i\pmod q$ for all $i$. 

After a computational search, Table~\ref{table:lens} shows, for each $3\leq n\leq 14$ and $3\leq q\leq 23$, a pair of $(2n-1)$-dimensional isospectral and non-isometric irreducible lens spaces with fundamental group of order $q$, in case it exists. 
It is worth mentioning that, for each choice of $n,q$ in the table, there exist more isospectral pairs than those appearing there, and the pair may belong to a larger isospectral family. 
The convention is to show the minimal lens space with respect to the lexicographic order that is isospectral to some other non-isometric lens space. 

Combining Theorem~\ref{thm:increasingdimension} and Table~\ref{table:lens}, we obtain that there are pairs of  $(2n-1)$-dimensional isospectral and non-isometric lens spaces with volume strictly greater than $\vol(S^{2n-1})/24$, for many choices of $n$.
Such choices are precisely encoded by condition \eqref{eq:congruences} of Theorem~\ref{thm1:highestvolume-manifold}, whose proof follows.

Table~\ref{table:existencelens} summarizes the existence problem of a pair of $(2n-1)$-dimensional isospectral and non-isometric lens spaces with fundamental group of order $q$, for the same values of $n$ and $q$ considered in Table~\ref{table:lens}. 
In particular, it gives the largest volume of an isospectral pair of spherical space forms for every odd dimension $5\leq 2n-1\leq 27$, namely, $\frac{\vol(S^{2n-1})}{q}$ with $q=11,13,16,17, 11,11,13,13,11,11,11,13$ for $n=3,\dots,14$ respectively. 

We conclude the section with some open questions and remarks.

\begin{question}
Are isospectral and non-isometric spherical space forms of largest volume always lens spaces? 
\end{question}

\begin{remark}
To the best of the authors' knowledge, the smallest size of the fundamental groups of a pair of isospectral spherical space forms with non-cyclic fundamental group known so far is $275$ (see \cite[\S5]{Ikeda97}). 
In the notation of the proof of Lemma~\ref{lem:non-cyclic24}, this pair is realized by two irreducible fixed point free representations of the Type I group $H_5(11,25,3)$ that are non-equivalent by automorphisms of $H_5(11,25,3)$. 
The corresponding spherical space forms have dimension $9$. 
\end{remark}

\begin{remark}
In line with the previous remark, it is very feasible to improve Lemma~\ref{lem:non-cyclic24} by increasing the order to a number greater than $24$. 
However, this is not the main obstruction to prove a statement valid for all $n\geq3$. 
Indeed, the computational results in the search of isospectral lens spaces for $q\leq 64$ provided a much larger set of conditions to \eqref{eq:congruences} in such a way the only values omitted for $n\leq 30000$ are $5039, 6479, 22319, 23759, 28799$.

A very curious fact in the computational results evidences a negative answer for the question below. 
\end{remark}

\begin{question}
Are there isospectral and non-isometric $(2n-1)$-dimensional lens spaces with fundamental group of order $q$ satisfying $n\equiv -1\pmod{\frac{\varphi(q)}{2}}$?
\end{question}

\begin{remark}
Because of a computational search, we know the largest volume of an $(2n-1)$-dimensional pair of isospectral and non-isometric spherical space forms for every $n\leq 14$. The results of our computation are recorded in Table~\ref{table:existencelens}.

Table~\ref{table:existencelens} implies the existence of an $(21+5r)$-dimensional isospectral pair with volume $\frac{\vol(S^{2(11+5r)-1})}{11}$ for every $r\in\N$ (i.e.\ $n=11+5r$ and $q=11$), but it is not clear whether these examples have the largest volumes possible in their dimensions because our computational search was done only for $n\leq 14$. 

It might be possible to prove the non-existence of isospectral lens spaces with fundamental group of order $q\leq 10$ in an arbitrary odd dimension.  
In this case, we obtain the largest volume in infinitely many dimensions. 
More precisely, the largest volume of an $(2n-1)$-dimensional pair of isospectral and non-isometric spherical space forms would be $\frac{\vol(S^{2n-1})}{11}$ for all $n$ congruent to $1$, $2$ or $3$ modulo $5$ and $n\geq11$. 
\end{remark}

%%%%%%%%%%%%%%%%%%%%%%%%%%%
%%%%%%%%%%%%%%%%%%%%%%%%%%%
%%%%%%%%%%%%%%%%%%%%%%%%%%%
%%%%%%%%%%%%%%%%%%%%%%%%%%%
%%%%%%%%%%%%%%%%%%%%%%%%%%%
%%%%%%%%%%%%%%%%%%%%%%%%%%%
%%%%%%%%%%%%%%%%%%%%%%%%%%%
%%%%%%%%%%%%%%%%%%%%%%%%%%%
%%%%%%%%%%%%%%%%%%%%%%%%%%%
%%%%%%%%%%%%%%%%%%%%%%%%%%%
%%%%%%%%%%%%%%%%%%%%%%%%%%%
%%%%%%%%%%%%%%%%%%%%%%%%%%%
%%%%%%%%%%%%%%%%%%%%%%%%%%%
%%%%%%%%%%%%%%%%%%%%%%%%%%%

\section{Finite part spectrum} \label{sec:finitespectrum}

The goal of this section is to obtain several statements ensuring that, for a quotient of a CROSS (Compact Rank One Symmetric Space), a convenient finite part of the spectrum (not necessarily the first eigenvalues) is sufficient to obtain the full spectrum under some geometric conditions. 

Let $G$ be a compact Lie group and $K$ a closed subgroup of it. 
We endow the homogeneous space $X:=G/K$ with a normal metric, that is, a Riemannian metric induced by an $\Ad(G)$-invariant inner product $\innerdots$ on the Lie algebra $\fg$ of $G$ (e.g.\ minus Killing form of $\fg$ provided $G$ is semisimple). 

For $\Gamma$ a discrete (hence finite) subgroup of $G$, the right regular representation of $G$ on $L^2(\Gamma\ba G)$ decomposes as 
\begin{equation}
L^2(\Gamma\ba G)\simeq \bigoplus_{(\pi,V_\pi)\in\widehat G}n_{\Gamma}(\pi)\, V_\pi.
\end{equation}
Unlike the case of non-compact semisimple Lie groups, the multiplicity $n_{\Gamma}(\pi)$ of $\pi$ in $L^2(\Gamma\ba G)$ can be easily determined by $n_{\Gamma}(\pi)=\dim V_\pi^\Gamma$. 
The spectrum $\Spec(\Gamma\ba X)$ of the Laplace-Beltrami operator $\Delta$ on $\Gamma\ba X$ is determined by these coefficients since the multiplicity of a non-negative real number $\lambda$ in $\Spec(\Gamma\ba X)$ is given by
\begin{equation}\label{eq:mult_Gamma(lambda)}
\mult_{\Gamma}(\lambda) 
:=\sum_{(\pi,V_\pi)\in\widehat G\, \mid\,  \lambda(C,\pi)=\lambda} n_{\Gamma}(\pi)\, \dim V_\pi^K, 
\end{equation}  
where $\lambda(C,\pi)$ is the scalar for which the Casimir operator $C$ associated to $(\fg,\innerdots)$ acts on $V_\pi$.
The value of $\lambda(C,\pi)$ can be computed in terms of the highest weight of $\pi$ via Freudenthal's formula. 
Note that the sum in \eqref{eq:mult_Gamma(lambda)} is restricted to $\widehat G_K:=\{\pi\in\widehat G: V_\pi^K\neq0\}$, the \emph{spherical representations} of the pair $(G,K)$.

It follows from \eqref{eq:mult_Gamma(lambda)} that $\Spec(\Gamma\ba X)$ is included in $\Spec(X)$ in the sense that every eigenvalue $\lambda$ in $\Spec(\Gamma\ba X)$ is necessarily in $\Spec(X)$ and $\mult_{\Gamma}(\lambda)\leq \mult(\lambda)$, where $\mult(\lambda)$ stands for the multiplicity of $\lambda$ in $\Spec(X)$. 

From now on we assume that $X$ is a simply connected CROSS realized as in \eqref{eq:CROSSrealizations}, that is, $S^{n}=\SO(n+1)/\SO(n)$, $P^n(\C)=\SU(n+1)/\op{S}(\Ut(n)\times\Ut(1))$, $P^n(\Hy)=\Sp(n+1)/\Sp(n)\times\Sp(1)$, $P^2(\mathbb O)=\op{F}_4/\Spin(9)$.  
We endow $X$ with the symmetric metric such that the sectional curvature is constantly one for spheres and satisfies $1\leq \op{sec}\leq 4$ for the rest of the cases.  
It turns out that the spherical representations of $(G,K)$ are given by a single string of representations 
\begin{equation}\label{eq:hatG_k}
\widehat G_K=\{(\pi_k,V_k) :k\geq0\}
\end{equation}
with $\dim V_k^K=1$ and that
\begin{equation}\label{eq:lambda_k}
\lambda_k:=\lambda(C,\pi_k) =
\begin{cases}
k(k+n-1) &\quad\text{if }X=S^n,\\
k(k+n) &\quad\text{if }X=P^n(\C),\\
k(k+2n+1) &\quad\text{if }X=P^n(\Hy),\\
k(k+11) &\quad\text{if }X=P^2(\mathbb O),\\
\end{cases}
\end{equation}
for every $k\geq0$. 
For a proof, see for instance Lemmas~4.1, 5.1, 6.2, 7.2, and 8.1 in \cite{LM-repequiv2}, where $\tau$ is always the trivial representation of $K$, so its highest weight is $0$ (see also \cite{Lauret-spec0cyclic}).  
It follows that the spectrum of $\Gamma\ba X$ is given by the eigenvalues $\lambda_k$ having multiplicity $n_{\Gamma}(\pi_k)$, that is, 
\begin{equation}\label{eq:Spec(X/Gamma)}
\Spec(\Gamma\ba X)= \Big\{\!\!\Big\{
\underbrace{\lambda_k,\dots,\lambda_k}_{n_{\Gamma}(\pi_k)}
:k\geq0
\Big\}\!\!\Big\}. 
\end{equation}

Let $|\Phi^+|$ denote the number of positive root associated to the complexified Lie algebra $\fg_\C$. 
One has $|\Phi^+|= n^2, n(n-1), \frac{n(n+1)}{2}, (n+1)^2, 24$ for $X=S^{2n},S^{2n-1}, P^n(\C), P^n(\Hy), P^2(\mathbb O) $ respectively.

The next statement is \cite[Thm.~1.2]{LM-strongmultonethm} applied to our case of interest. 

\begin{proposition}\label{prop:strongmultone}
Let $G$ be a classical compact Lie group or the compact simple Lie group $\op{F}_4$, 
and let $\{\pi_k:k\geq0\}$ denote the irreducible representations of $G$ as in \eqref{eq:hatG_k}. 
Given a positive integer $q$, let $\mathcal A$ be any finite subset of $\N_0$ satisfying that  
\begin{equation}\label{eq:finitecondition}
|\mathcal A\cap (j+q\Z)|\geq |\Phi^+|+1\quad\text{ for all }0\leq j\leq q-1.
\end{equation}
Then, for any finite subgroup $\Gamma$ of $G$ with $q$ divisible by $|\Gamma|$, the finite set of multiplicities $n_{\Gamma}(\pi_k)$ for $k\in\mathcal A$ determine $n_{\Gamma}(\pi_k)$ for all $k\geq0$. 
In particular, if $\Gamma_1,\Gamma_2$ are finite subgroups of $G$ with $q$ divisible by $|\Gamma_i|$ for $i=1,2$ such that $n_{\Gamma_1}(\pi_k)=n_{\Gamma_2}(\pi_k)$ for all $k\in\mathcal A$, then $n_{\Gamma_1}(\pi_k)=n_{\Gamma_2}(\pi_k)$ for all $k\geq0$. 
\end{proposition}

The first step in the proof (see \cite[Prop.~2.2]{LM-strongmultonethm}) is to show that there is a polynomial $p_\Gamma(z)$ of degree less than $q(|\Phi^+|+1)$ such that
\begin{equation}
F_{\Gamma}(z):= \sum_{k\geq0} n_{\Gamma}(\pi_k)\,z^k 
=\frac{p_{\Gamma}(z)}{(1-z^q)^{|\Phi^+|+1}}. 
\end{equation}
The second one (see \cite[Prop.~3.1]{LM-strongmultonethm}) is to expand the right hand side of the above identity and, from the identities corresponding to the $k$-th term for $k\in\mathcal A$, obtain the values of all coefficients of $p_{\Gamma}(z)$ in terms of $n_{\Gamma}(\pi_k)$ for $k\in\mathcal A$. 

Here is an immediate consequence. 

\begin{corollary}
Let $X$ be a simply connected compact rank one symmetric space realized as in \eqref{eq:CROSSrealizations}
Given a positive integer $q$, let $\mathcal A$ be as in Proposition~\ref{prop:strongmultone}. 
Then, for any finite subgroup $\Gamma$ of $G$ with $q$ divisible by $|\Gamma|$, the finite set of multiplicities $\mult_{\Gamma}(\pi_k)$ of $\lambda_k$ in $\Spec(\Gamma\ba X)$ for each $k\in\mathcal A$ determine the spectrum $\Spec(\Gamma\ba X)$ of the orbifold $\Gamma\ba X$.
In particular, if $\Gamma_1,\Gamma_2$ are finite subgroups of $G$ with $q$ divisible by $|\Gamma_i|$ for $i=1,2$ such that $\mult_{\Gamma_1}(\pi_k) = \mult_{\Gamma_2}(\pi_k)$ for all $k\in\mathcal A$, then the orbifolds $\Gamma_1\ba X$ and $\Gamma_2\ba X$ are isospectral. 
\end{corollary}

We are now in position to prove the main theorem of this section. 

\begin{proof}[Proof of Theorem~\ref{thm1:finitespectrum}]
We set $q=\lfloor \frac{1}{\varepsilon}\rfloor!$ and 
\begin{align}
N&=1+\sum_{k=0}^{q(|\Phi^+|+1)} \dim V_k.  
\end{align}

Let $\Gamma_1,\Gamma_2$ be finite subgroups of $G$ such that $q_i:=|\Gamma_i|\leq \frac{1}{\varepsilon}$ for $i=1,2$, thus $q_i$ divides $q$ and $\frac{\vol(\Gamma_i\ba X)}{\vol(X)}>\varepsilon$.
The converse is obvious, so we assume that the first $N$ eigenvalues of $\Gamma_1\ba X$ and $\Gamma_2\ba X$ coincide. 

We write $\Spec(\Gamma_i\ba X)$ as $0=\lambda_0(\Gamma_i\ba X)< \lambda_1(\Gamma_i\ba X)\leq\lambda_2(\Gamma_i\ba X)\leq \dots\leq \lambda_j(\Gamma_i\ba X)\leq \dots$ counted with multiplicities. 
By \eqref{eq:Spec(X/Gamma)}, it follows that 
\begin{equation}\label{eq:reasoning}
\begin{aligned}
\lambda_0&=\lambda_0(\Gamma_i\ba X),\\
\lambda_1&=\lambda_j(\Gamma_i\ba X)
	\quad \text{ for }1\leq j\leq n_{\Gamma_i}(\pi_1),\\
\lambda_2&=\lambda_j(\Gamma_i\ba X)
	\quad \text{ for }1\leq j-n_{\Gamma_i}(\pi_1)\leq n_{\Gamma_i}(\pi_2),\\
&\;\;\vdots \\
\lambda_k&=\lambda_j(\Gamma_i\ba X)
	\quad \text{ for }1\leq j-{\textstyle \sum\limits_{l=1}^{k-1}} n_{\Gamma_i}(\pi_l)\leq n_{\Gamma_i}(\pi_k).
\end{aligned}
\end{equation}
Recall that $\lambda_k$ is explicitly given in \eqref{eq:lambda_k} for any $k\geq0$. 
Now, since 
\begin{equation}
\lambda_j(\Gamma_1\ba X)=\lambda_j(\Gamma_2\ba X)
\quad\forall\, j=0,\dots, N-1
=\sum\limits_{k=0}^{q(|\Phi^+|+1)} n_{\Gamma_i}(\pi_k),
\end{equation}
we obtain that $n_{\Gamma_1}(\pi_k)=n_{\Gamma_2}(\pi_k)$ for all $k=0,\dots, q(|\Phi^+|+1)$.
Clearly, $\mathcal A:=\{k\in\N_0: k\leq q(|\Phi^+|+1)\}$ satisfies \eqref{eq:finitecondition}. 
Since $q_1$ and $q_2$ divide $q$, Proposition~\ref{prop:strongmultone} implies that $n_{\Gamma_1}(\pi_k)= n_{\Gamma_2}(\pi_k)$ for all $k\geq0$, therefore $\Spec(\Gamma_1\ba X) = \Spec(\Gamma_2\ba X)$. 
\end{proof}

The next example shows that the condition $\frac{\vol(\Gamma_i\ba X)}{\vol(X)}>\varepsilon$ cannot be omitted in Theorem~\ref{thm1:finitespectrum}. 

\begin{example}\label{ex:omittingvolume}
Let $N$ be an arbitrary positive integer. 
We will show that there are $3$-dimensional lens orbifolds whose first $N$ eigenvalues coincide but which are not isospectral. 

We set $X=S^3$, thus $G=\SO(4)$, $n=2$.
Let $q$ be the smallest integer divisible by $4$ such that 
\begin{equation}
N\leq 
\frac{q(q+4)}{16}=
\sum_{k=0}^{\frac q2-1} \big(\lfloor\tfrac k2\rfloor+1\big)
. 
\end{equation}  
For any divisor $d$ of $q$, let $L_d=L(q;0,d)$ and $\Gamma_d=\Gamma_{q;(0,d)}$, so $L_d=S^3/\Gamma_d$. 
It is evident from Proposition~\ref{prop:lens-isometries} that these lens orbifolds are pairwise non-isometric.

One can easily check that the associated congruence lattice $\mathcal L_d$ of $L_d$ (see \eqref{eq:congruencelattice}) is given by 
\begin{equation}
\mathcal L_d=\{(a,b)\in\Z^2: \tfrac{q}{d}\mid b\} =\Z\times\tfrac{q}{d}\Z. 
\end{equation}
This implies that they are mutually non-isospectral by \eqref{eq:F_L-thetafunction}. 

We claim that the first $N$ eigenvalues of the Laplacians on $L_1$ and $L_2$ coincide. 
Since $\mathcal L_1=\Z\times q\Z$ and $\mathcal L_2=\Z\times \tfrac{q}{2}\Z$, it is clear that $N_{\mathcal L_1}(k)=N_{\mathcal L_2}(k)=1$ for all $k<\tfrac q2$. 
Now, \eqref{eq:dimH_k^Gamma-onenorm} immediately implies that
\begin{equation}
n_{\Gamma_i}(\pi_k)=\dim V_{\pi_k}^{\Gamma_i}
=\lfloor\tfrac k2\rfloor+1
\qquad\forall \,k<\tfrac q2, \text{ for $i=1,2$.}
\end{equation} 
The same reasoning in \eqref{eq:reasoning} tells us that $\lambda_j(L_1)=\lambda_j(L_2)$ for all $j$ satisfying
\begin{equation}
0\leq j\leq 
\sum_{k=0}^{\frac q2-1} n_{\Gamma_i}(\pi_k)
=\sum_{k=0}^{\frac q2-1} \big(\lfloor\tfrac k2\rfloor+1\big)
\end{equation}  
and the assertion follows. 

\end{example}

\begin{remark}
Although we have considered only the Laplace-Beltrami operator acting on functions on $\Gamma\ba X$, the discussion can be extended to the standard Laplacian acting on smooth sections of homogeneous vector bundles of CROSSes. 
Among them, we have the Hodge-Laplace operator acting on smooth $p$-forms, and the Lichnerowicz Laplacian acting on (symmetric) $k$-tensors.
\end{remark}

\begin{remark}
Can Theorem~\ref{thm1:finitespectrum} be extended to some (or all) compact symmetric spaces of rank at least two? 
The main difficulty for such a symmetric space $G/K$ is that $\widehat G_K$ cannot be written as a finite union of strings, in the sense of \cite[Def.~2.1]{LM-strongmultonethm} (see also \cite[Rmks.~5.3 and 5.4]{LM-strongmultonethm}). 
\end{remark}

\section{Isospectral towers of lens spaces} \label{sec:towers}

In this section we will construct isospectral towers of lens spaces in every odd dimension $\geq 5$. Our construction will make extensive use of some terminology defined by Doyle and DeFord~\cite{DeFordDoyle18}. 
Let $r> 2, t\geq 1$ be positive integers.

\begin{definition} We say that $a\in\Z^n$ is 
\begin{itemize}
\item {\it univalent mod $r$} if its entries are distinct mod $r$,
\item {\it reversible mod $r$} if there exists $c\in\Z_r$ such that $a+c$ and $-a$ are equal as multisets mod $r$,
\item {\it good mod $r$} if it is univalent or reversible mod $r$,
\item {\it hereditarily good mod $r$} if it is good mod $d$ for all divisors $d$ of $r$, and 
\item {\it useful mod $r$} if it is hereditarily good and irreversible mod $r$.
\end{itemize}
\end{definition}

\begin{remark}
Every $a\in \Z^n$ is reversible (and hence good) mod $1,2$. So to check whether or not $a$ is hereditarily good we need only consider divisors $d$ of $r$ for which $d>2$. In particular, when $r$ is prime $a\in\Z^n$ is hereditarily good mod $r$ if and only if it is good mod $r$.
\end{remark}

\begin{example}
Let $a=(1,3,6)$. Then $a$ is 
\begin{itemize}
\item univalent mod $r$ when $r\not\in \{1,2,3,5\}$,
\item reversible mod $r$ when $r\in \{1,2,4,7,8\}$,
\item good mod $r$ when $r\neq 3,5$,
\item hereditarily good mod $r$ when $r$ is not divisible by $3$ or $5$, and
\item useful mod $r$ for any $r\geq 11$ not divisible by $3$ or $5$. 
\end{itemize}
\end{example}

Given $r,t,a$ as above, we define the lens space $\LMR(r,t,a)$ by \[\LMR(r,t,a) = L(r^2t; rta_1+1, \dots, rta_n+1).\]
The main theorem of \cite{DeFordDoyle18} is:

\begin{theorem}\label{thm:Doyle}
If $a\in\Z^n$ is hereditarily good mod $r$, then for any $t$ the lens spaces $\LMR(r,t,a)$ and $\LMR(r,t,-a)$ are isospectral.
\end{theorem}

\begin{remark}
The lens spaces $\LMR(r,t,a)$ and $\LMR(r,t,-a)$ in Theorem~\ref{thm:Doyle} are actually \emph{$p$-isospectral for all $p$}.
This means that the Hodge-Laplace operator acting on $p$-forms on each of them has the same spectrum for all $p=0,\dots,2n-1$. 
\end{remark}

Proposition~\ref{prop:lens-isometries} tells us that $a$ is reversible mod $r$ precisely when the lens spaces $\LMR(r,t,a)$ and $\LMR(r,t,-a)$ are isometric and hence trivially isospectral. The theorem is thus only interesting in the case that $a$ is hereditarily good and irreversible mod $r$. This is what prompted DeFord and Doyle to call such $a$ {\it useful}: when $a$ is useful mod $r$ the lens spaces $\LMR(r,t,a)$ and $\LMR(r,t,-a)$ will be isospectral but not isometric.

\begin{definition} We say that $a\in\Z^n$ is self-reversing mod $r$ if $a$ and $-a$ are equal as multisets mod $r$.
\end{definition}

\begin{lemma}\label{lemma:sumtozero}
Let $n\geq 3$ and $r>2$ be coprime positive integers and $a\in \Z^n$. If $a$ is not self-reversing mod $r$ and the sum of its entries is divisible by $r$, then $a$ is irreversible mod $r$.
\end{lemma}
\begin{proof}
If $a$ is reversible then there exists $c\in\Z_r$ such that $a+c$ and $-a$ are equal as multisets mod $r$. In this case for every $i\in \{1,\dots,m\}$ there exists $j\in\{1,\dots, m\}$ such that \[a_i+c\equiv -a_j\pmod{r}.\] If we sum over all $i$ then we get \[2S+nc\equiv 0 \pmod{r},\] where $S=\sum_{i=1}^n a_i$. By hypothesis $S\equiv 0\pmod{r}$, hence $nc\equiv 0 \pmod{r}$. Since $n$ is coprime to $r$, it is invertible modulo $r$ and consequently we find that $c\equiv 0 \pmod{r}$. But this means that $a$ is self-reversing mod $r$, a contradiction.
\end{proof}

Ikeda's families of isospectral lens spaces show that there are isospectral lens spaces of arbitrarily high dimension. Although it is well known to experts that there are in fact isospectral lens spaces in every odd dimension at least $5$, this result has not, to our knowledge, ever appeared in the literature. Below we provide a proof.

\begin{proposition}\label{prop:lenseverydimension}
In every odd dimension at least $5$, there are pairs of non-isometric lens spaces that are isospectral (and $p$-isospectral for all $p$).
\end{proposition}

\begin{proof}
We fix $n\geq3$. 
Let $r>n^2$ be a prime number and observe that \[a=\left(1,2,\dots, n-1, r-\frac{n(n-1)}{2}\right)\] is not self-reversing mod $r$ (since $a_1=1$ and no entry of $a$ is equal to $r-1$) and  satisfies that the sum of its entries is $0$ mod $r$ since this sum will in fact be equal to $r$. It follows from Theorem~\ref{thm:Doyle} that the lens spaces $\LMR(r,t,a)$ and $\LMR(r,t,-a)$ will be isospectral but not isometric for all $t$.
\end{proof}

It is interesting to note that when $\gcd(n,r)=1$, all tuples $a\in\Z^n$ which are useful mod $r$ yield lens spaces isometric to those produced by tuples whose sum of entries is congruent to $0$ mod $r$.

\begin{lemma}
Let $n\geq 3$ and $r>2$ be coprime positive integers and $t>1$ be any positive integer. If $a\in \Z^n$ is useful mod $r$ then there exists $b\in\Z^n$ such that:
\begin{enumerate}
\item $b$ is useful mod $r$,
\item the sum of the entries of $b$ is $0$ mod $r$, and
\item the lens spaces $\LMR(r,t,a)$ and $\LMR(r,t,b)$ are isometric.
\end{enumerate}
\end{lemma}
\begin{proof}
Let $c\in\Z_r$. The sum of the entries of $a+c=(a_1+c,\dots, a_n+c)$ is $S+nc$ where $S$ is the sum of the entries of $a$. Because $n$ is coprime to $r$, it is invertible modulo $r$. Let $m\in \Z$ be such that $mn\equiv 0 \pmod{r}$. A straightforward argument now shows that if $c\equiv -mS\pmod{r}$ then the sum of the entries of $a+c$ will be $0$ mod $r$. Defining $b=a+c$ we see that (2) and (3) are trivially satisfied. To prove (1), note that Lemma \ref{lemma:sumtozero}  shows that $b$ is irreversible, while $b$ is hereditarily good mod $r$ since $a$ is hereditarily good mod $r$ and $b=a+c$.
\end{proof}

\begin{theorem}
Let $n\geq 3$. There exist infinitely many pairs of descending isospectral towers of lens spaces in dimension $2n-1$.
\end{theorem}
\begin{proof}
We begin with the example from the proof of Proposition~\ref{prop:lenseverydimension}:
\[a=\left(1,2,\dots, n-1, r-\frac{n(n-1)}{2}\right),\] where $r$ is a fixed prime greater than $n^2$. As was mentioned earlier, for any $t$ the lens spaces $\LMR(r,t,a)$ and $\LMR(r,t,-a)$ are isospectral but not isometric. Let $k$ be any positive integer greater than $1$ satisfying $k\equiv 1 \pmod {r}$ and define $t_j=tk^j$. Thus $\LMR(r,t_j,a)$ and $\LMR(r,t_j,-a)$ are isospectral and not isometric for all $j\geq 0$.

\begin{nnclaim}
For any $i>j\geq 0$ the lens spaces $L(r^2t_j; rt_ia_1+1, \dots, rt_ia_n+1)$ and $\LMR(r,t_j,a)$ are isometric.
\end{nnclaim}
\begin{proof}
This follows from Proposition~\ref{prop:lens-isometries} since $rt_ia_s+1\equiv rt_ja_s+1\pmod{r^2t_j}$ for $s=1,\dots,n$, which in turn follows from the fact that $k\equiv 1 \pmod{r}$, as we now show. That $k\equiv 1 \pmod{r}$ and $i>j$ implies $k^{i-j}-1$ is divisible by $r$. Thus there is an integer $m$ such that $k^{i-j}-1=rm$. Then 
\begin{align*}
(rt_i a_s + 1) - (rt_ja_s+1) 
%&= (rtk^i a_s +1) - (rtk^j a_s +1) \\
&= rtk^j a_s (k^{i-j}-1) 
%&= rt_ja_s rm \\
= r^2t_j a_s m.  
\end{align*} 
In particular $r^2t_j$ divides $(rt_i a_s + 1) - (rt_ja_s+1)$ and $rt_ia_s+1\equiv rt_ja_s+1\pmod{r^2t_j}$\ as claimed.
\end{proof}

We now define our towers inductively. Let $M_0=\LMR(r,t_0,a)$ and $N_0=\LMR(r,t_0,-a)$. We have already seen that $M_0$ and $N_0$ are isospectral and not isometric. To construct $M_1$ and $N_1$, we note that the claim shows that $L(r^2t_0; rt_1a_1+1, \dots, rt_1a_n+1)$ and $\LMR(r,t_0,a)$ are isometric. Since the former is a finite degree cover of $\LMR(r,t_1,a)$, there is some lens space $M_1$ isometric to $\LMR(r,t_1,a)$ having $\LMR(r,t_0,a)$ as a finite degree cover. Similarly, we obtain a lens space $N_1$ isometric to $\LMR(r,t_1,-a)$ having $\LMR(r,t_1,-a)$ as a finite degree cover. Since isometric manifolds are trivially isospectral, the isospectrality of $M_1$ and $N_1$ follows from the isospectrality of $\LMR(r,t_1,a)$ and $\LMR(r,t_1,-a)$. 

Suppose now that we have constructed $M_{i-1}$ and $N_{i-1}$, and note that by construction these lens spaces will be isometric to $\LMR(r,t_{i-1},a)$ and $\LMR(r,t_{i-1},-a)$. By the claim, the lens space $L(r^2t_{i-1}; rt_ia_1+1, \dots, rt_ia_n+1)$ is isometric to $\LMR(r,t_i,a)$ and hence to $M_{i-1}$ as well. Therefore there is a lens space $M_i$ isometric to $\LMR(r,t_i,a)$ having $M_{i-1}$ as a finite degree cover. We similarly obtain a lens space $N_i$ isometric to $\LMR(r,t_i,-a)$ having $N_{i-1}$ as a finite degree cover. As above, $M_i$ and $N_i$ are isospectral but not isometric. Continuing in this manner yields the desired pair of descending isospectral towers $\{M_i\}$ and $\{N_i\}$.

It remains only to show that there are infinitely many such pairs of descending towers. This follows immediately from the fact that the above construction holds for any positive integer $t$. We can therefore obtain infinitely many towers by allowing $t$ to range over the set of prime numbers not equal to $r$ and not dividing $k$. 
\end{proof}

%%%%%%%%%%%%%%%%%%%%
%%%%%%%%%%%%%%%%%%%%
%%%%%%%%%%%%%%%%%%%%
%%%%%%%%%%%%%%%%%%%%
%%%%%%%%%%%%%%%%%%%%

\section{Isospectrality between quotients of symmetric spaces} \label{sec:McReynolds}

In this section we discuss Question~\ref{prob:symspaces-manifolds}, which asks for the determination of the compact simply connected irreducible symmetric spaces that cover isospectral and non-isometric manifolds.  
Henceforth we will omit the fact that the isospectral covered manifolds are non-isometric, for the sake of conciseness. 

Unlike the non-compact type setting, the structure of locally symmetric spaces of compact type is quite rigid. 
In particular, the classification of manifolds covered by spheres, the so called spherical space forms, was a long process finished by Wolf in \cite{Wolf-book}. 
Wolf also classified in \cite[Ch.~9]{Wolf-book} the manifolds covered by compact symmetric spaces $G/K$ satisfying $\op{rank}(G)-\op{rank}(K)\leq 1$.  
We will use his results throughout the section. 

We now start discussing partial answers of Question~\ref{prob:symspaces-manifolds}, starting from the simplest case. 
There are many compact irreducible symmetric spaces that cannot cover isospectral pairs because they do not cover any manifold at all. 
For instance, this is the case for 
even dimensional real projective spaces $P^{2n}(\R)= \frac{\SO(2n+1)}{\Ot(2n)}$, 
quaternionic Grassmannian spaces $\frac{\Sp(m+n)}{\Sp(m)\times\Sp(n)}$ for $m>n\geq1$, 
complex Grassmannian spaces $\frac{\SU(m+n)}{\op{S}(\Ut(m)\times\Ut(n))}$ for $m>n\geq1$ with $mn$ even,
$\frac{\textup{E}_6}{\SU(6)\cdot \SU(2)}$,
and $\frac{\textup{E}_6}{\SO(10)\cdot \SO(2)}$
Similarly, the real Grassmannian spaces $\frac{\SO(m+n)}{\SO(m)\times\SO(n)}$ with $m+n$ odd, cover properly exactly one manifold with fundamental group of order $2$, hence these spaces do not cover isospectral manifolds.
The best known instance are the even dimensional spheres $S^{2n}=\frac{\SO(2n+1)}{\SO(2n)}$ which only cover properly $P^{2n}(\R)$.

Even dimensional Grassmannian spaces cover only finitely many manifolds. 
On the opposite side, odd dimensional Grassmannian spaces (i.e.\ $\frac{\SO(m+n)}{\SO(m)\times\SO(n)}$ with $mn$ odd) cover infinitely many manifolds. 
Ikeda is the main contributor to Question~\ref{prob:symspaces-manifolds} for these spaces. 
We summarize his results \cite[Thm.~II]{Ikeda80_3-dimI}, \cite[Thm.~4]{Ikeda83}, and \cite[Thm.~7]{Ikeda97}, in the following statement.

\begin{theorem}[Ikeda] \label{thm:Ikeda-existence}
Odd-dimensional spheres of dimension $\geq5$ and real Grassmannian spaces $\frac{\SO(m+n)}{\SO(m)\times\SO(n)}$ with $n\geq m>1$ satisfying $mn\equiv 1\pmod 2$ and $m+n\in\{2k:k=5\text{ or }k\geq7\}$, cover infinitely many isospectral and non-isometric pairs of manifolds. 
Furthermore, the $3$-dimensional sphere does not cover any pair of isospectral and non-isometric manifolds. 
\end{theorem}

Wolf~\cite{Wolf01} extended Ikeda's isospectral constructions by adopting a much more general perspective.

\begin{remark}
Ikeda proved the existence of infinitely many pairs of (non-cyclic) almost conjugate but not conjugate subgroups of $\SO(2d)$ for $d=5$ or $d\geq7$ acting freely on $S^{2d-1}=\frac{\SO(2d)}{\SO(2d-1)}$ and $\frac{\SO(2d)}{\SO(m)\times\SO(2d-m)}$ for any $m$ odd. 
This proves Theorem~\ref{thm:Ikeda-existence} for all cases excepting $S^{5}$, $S^{7}$, and $S^{11}$, where he proved the existence of \emph{finitely many} isospectral lens spaces covered by each of them. 
The existence of infinitely many pairs of isospectral lens spaces of dimension $5$, $7$, and $11$ follows from \cite{DeFordDoyle18} (see also Section \ref{sec:towers}).
\end{remark}

\begin{remark}
Even dimensional Grassmannian spaces (i.e.\ $\frac{\SO(m+n)}{\SO(m)\times\SO(n)}$ with $mn$ even and all complex and quaternionic Grassmannian spaces) cover finitely many manifolds; see \cite[\S9.3]{Wolf-book} for the classification. 
Spectrally distinguishing them may a feasible but tedious achievement. 
A similar situation should occur with $\frac{\SO(2n)}{\Ut(n)}$, $\frac{\Sp(n)}{\Ut(n)}$, $\frac{\op{E}_7}{(\SU(8)/\{\pm I\})}$, $\frac{\op{E}_7}{\op{E}_6\cdot \op{T}^1}$ (see \cite[\S{}9.4]{Wolf-book}). 

A more challenging problem is to decide whether the symmetric space $\SU(3)/\SO(3)$ covers isospectral manifolds.
Its corresponding locally symmetric spaces were classified by Wolf~\cite[Lem.~9.6.3]{Wolf-book} and turn out to be infinitely many manifolds with cyclic fundamental groups. 
\end{remark}

The existence of isospectral covered manifolds is more feasible for compact irreducible symmetric spaces of \emph{group type}. 
These are of the form $\frac{K\times K}{\diag(K)}$, where $K$ is a compact simple Lie group and $\diag(K)=\{(k,k)\in K\times K: k\in K\}$.
We will abbreviate $\frac{K\times K}{K}=\frac{K\times K}{\diag(K)}$.
Alternatively, $\frac{K\times K}{K}$ is isometric to $K$ endowed with a bi-invariant metric. 

The following consequence of Sunada's method provides many examples.

\begin{proposition}\label{prop:grouptype}
Let $K$ be a compact connected simple Lie group and let $g_{0}$ denote any bi-invariant metric on $K$ (which is unique up to a positive multiple). 
If $\Gamma_1$ and $\Gamma_2$ are almost conjugate but non-conjugate finite subgroups of $K$, then $(K/\Gamma_1,g_1)$ and $(K/\Gamma,g_2)$ are isospectral and non-isometric, where $g_i$ denotes the Riemannian metric on $K/\Gamma_1$ induced by $g_0$ (i.e.\ $(K,g_0)\to (K/\Gamma_i,g_i)$ is a Riemannian cover). 
\end{proposition} 

For a proof, see \cite[Prop.~2.10]{Wolf01} and take into account that $\Gamma_i$ acts freely on $K$.

\begin{proof}[Proof of Theorem~\ref{thm1:symectricspaces}]
By Proposition~\ref{prop:grouptype}, it is sufficient to show the existence of almost conjugate and non-conjugate subgroups in any simply connected compact simple Lie group $K$ non-isomorphic to $\SU(2)$, $\SU(3)$, $\Sp(2)$ or $\textrm{G}_2$. 

The pair $\Gamma_1,\Gamma_2\subset \SO(6)$ 
from \eqref{eq:almostconjugate-in-SO(6)} gives almost conjugate and non-conjugate subgroups of the universal cover $\Spin(6)\simeq \SU(4)$ of $\SO(6)$.
For any $n\geq7$, the image of $\Gamma_1,\Gamma_2$ under the canonical embedding $\Spin(6)\hookrightarrow \Spin(n)$ are still almost conjugate and non-conjugate in $\Spin(n)$.  
The same occurs for the well-known embeddings  
\begin{equation}
\begin{aligned}
\SU(4)&\hookrightarrow \SU(n)
	\quad\text{for all }n\geq5,\\
\SU(4)&\hookrightarrow \Ut(4)\hookrightarrow\Sp(n)
	\quad\text{for all }n\geq4,\\
\SO(6)&\hookrightarrow\SO(10)\hookrightarrow \op{E}_6\hookrightarrow \op{E}_7\hookrightarrow \op{E}_8,\\
\Spin(6)&\hookrightarrow\Spin(9)\hookrightarrow \op{F}_4,
\end{aligned}
\end{equation}
and the proof is complete.
\end{proof}

The symmetric space associated to $\SU(2)\simeq\Spin(3)\simeq\Sp(1)$ is isometric to $S^3$, hence it does not cover isospectral manifolds by Theorem~\ref{thm:Ikeda-existence}. 
It is worth to mention as a similar result that Vásquez~\cite{Vasquez18} proved that two almost conjugate subgroups in $\Spin(4)\simeq\SU(2)\times\SU(2)$ are necessarily conjugate.

\begin{question}
Are there almost conjugate and non-conjugate subgroups in $\SU(3)$ (resp.\ $\Sp(2)\simeq \Spin(5)$, $\Sp(3)$ and $\op{G}_2$)? 
\end{question}

\begin{question}
Which compact symmetric spaces of group type (not necessarily simply connected) cover isospectral manifolds. 
\end{question}

\begin{problem}
Study Question~\ref{prob:symspaces-manifolds} in the context of orbifolds covered by compact symmetric spaces. 
\end{problem}

\section{Open problems and questions}\label{sec:openquestions}

In this section we discuss further problems and questions, in addition to those included in the previous sections.

%{\red 
%--- in the inverse direction of asking from the compact type to the non-compact type, we may ask whether there are isospectral but not strongly isospectral locally symmetric spaces of non-compact type. Such examples exist in the compact type case due to Lens spaces by Ikeda
%
%--- some question about homology things from Bartel-Page's articles?: \href{https://arxiv.org/abs/1605.04866}{first article}, and \href{https://arxiv.org/abs/1803.06903}{second article}.
%
%--- some question about $\Z$-equivalence from Dipendra Prasad's article?
%
%--- some question from Prasad-Rapinchuk's papers?
%
%--- some question from Rajan's papers?
%
%--- For a fixed $n$, is there $C=C(n)$ such that any family of isospectral lens spaces (or spherical space forms) has at least $C$ elements. 
%
%--- length spectra of spherical space forms 
%
%--- Can a lens space be isospectral to a spherical orbifold which is not a lens space? NOOO
%
%}

\subsection{Constructing large families of spherical space forms}%isospectral manifolds}

Call a set of Riemannian manifolds an {\it isospectral set} if the manifolds in the set are pairwise isospectral and non-isometric. In \cite{BGG} Brooks, Gornet, and Gustafson used Sunada's method in order to construct isospectral sets of hyperbolic surfaces of arbitrarily large cardinality. In particular, they constructed an infinite sequence of natural numbers $g_i$ such that for each $i$ there is an isospectral set of genus $g_i$ hyperbolic surfaces of cardinality at least $g_i^{c \log(g_i)}$, for some constant $c$ which does not depend on $i$. We note that these isospectral sets are the largest known for hyperbolic surfaces. 

The Brooks-Gornet-Gustafson construction was subsequently generalized by McReynolds~\cite{McReynolds14} to higher dimensional real hyperbolic spaces, to the complex hyperbolic $2$-space, and to the symmetric spaces of arbitrary non-compact simple Lie groups. Later work of Belolipetsky and the second author \cite{BL} extended the Brooks-Gornet-Gustafson construction to simple Lie groups of real rank at least $2$.

In light of the aforementioned work it is natural to ask for the maximal size of an isospectral set of spherical manifolds or orbifolds. Every finite group admits an orthogonal representation, hence the results mentioned above (which all employ Sunada's method) can be modified so as to produce isospectral sets of spherical orbifolds with arbitrarily large cardinalities. It should be noted however, that the dimension of the manifolds in these isospectral sets will go to infinity as the cardinality does.

\begin{problem}\label{problem:problemconstructinghugefamilies}
Fix a dimension $d$. What is the maximal cardinality of an isospectral set of spherical space forms %manifolds 
of volume $x$ and dimension $d$?
\end{problem}

This problem seems particularly tractable in the setting of lens spaces. It seems reasonable, for example, to expect that one could modify Ikeda's proof \cite{Ikeda80_isosp-lens} of the existence of pairs of isospectral non-isometric lens spaces in order to obtain isospectral sets of larger cardinality.

\subsection{Upper bounds on the cardinality of an isospectral set}
In this section we discuss a problem which serves as a natural complement to Problem \ref{problem:problemconstructinghugefamilies}. Given a hyperbolic surface $S$, how many hyperbolic surfaces are there that are isospectral to $S$ but not isometric to it? In other words, what is the maximal cardinality of an isospectral set containing $S$? The Brooks-Gornet-Gustafson construction shows that in general, if $S$ has genus $g$, then there may be as many as $g^{c \log(g)}$ other hyperbolic  surfaces isospectral to $S$ but not isometric to it. The first upper bound for this quantity is due to Buser~\cite{Buser92}, who showed that if $g$ denotes the genus of $S$, then there are at most $e^{720 g^2}$ hyperbolic surfaces that are isospectral to $S$ but not isometric to it. This result was later improved to $e^{cg\log(g)}$ (for some universal constant $c$) by Parlier~\cite{Parlier}.

\begin{problem}
Let $M$ be a spherical orbifold of volume $V$. What is an upper bound for the cardinality of an isospectral set of spherical orbifolds which contains $M$?
\end{problem}

\subsection{Wolpert's genericity results}
Before Vign{\'e}ras~\cite{Vigneras80} constructed the first examples of isospectral Riemann surfaces, Wolpert~\cite{Wolpert77-Riemannsurfaces, Wolpert79-Riemannsurfaces} proved that a generic Riemann surface of genus $g\geq2$ is spectrally unique within the moduli space $\mathcal M_g$ of isometry classes of Riemann surfaces of genus $g$; that is, a generic Riemann surface is not isospectral to any non-isometric Riemann surface of the same genus. 

\begin{theorem}[Wolpert~\cite{Wolpert79-Riemannsurfaces}]
For each $g\in\N$, there is a dense subset $\mathcal M_g^\bullet$ of $\mathcal M_g$ satisfying that for any $S\in\mathcal M_g^\bullet$ one has that $\Spec(S')\neq\Spec(S)$ for all $S'\neq S$ in $\mathcal M_g$. 
\end{theorem}

In the moduli space $\mathcal T_n$ of $n$-dimensional flat tori, he proved the analogous result. 

\begin{theorem}[Wolpert~\cite{Wolpert78-tori}]
For each $n\in\N$, there is a dense subset $\mathcal T_n^\bullet$ of $\mathcal T_n$ satisfying that for any $T\in\mathcal T_n^\bullet$ one has that $\Spec(T')\neq\Spec(T)$ for all $T'\neq T$ in $\mathcal T_n$. 
\end{theorem}

It is natural to ask if a similar situation occurs for lens spaces. 

\begin{problem}
Is a ``generic'' lens space spectrally unique?
\end{problem}

Note that there is not a direct analogy to the case of Riemann surfaces and flat tori due to the absence of a natural topology on the space of lens spaces. In what follows we formulate a possible extension of Wolpert's results and provide numerical evidence for its validity.

For positive integers $n,q$, let us denote by $\lens(n,q)$ the isometry classes of $(2n-1)$-dimensional lens spaces with fundamental group of order $q$.
We set 
\begin{equation}
\unique(n,q) =\{L\in \lens(n,q): \Spec(L)\neq \Spec(L')\quad\forall \,L'\neq L\text{ in }\lens(n,q)\}
.
\end{equation}
In words, $\unique(n,q)$ is the subset of $\lens(n,q)$ given by lens spaces that are spectrally unique within $\lens(n,q)$.  

The space of all (isometry classes) of $(2n-1)$-dimensional lens spaces is $\bigcup_{q\geq1}\lens(n,q)$. 
Now, for a positive number $x$, it is natural to ask whether the density
\begin{equation}
\mathcal U_n(x): =\frac{\sum\limits_{q\leq x}\#\unique(n,q)}{\sum\limits_{q\leq x}\#\lens(n,q)}
\end{equation}
of the set of spectrally unique lens spaces with fundamental groups of order $\leq x$ into the set of lens spaces with fundamental groups of order $\leq x$ is close to $1$. 

\begin{conjecture}
One has $\displaystyle \lim_{x\to\infty} \mathcal U_n(x)=1$ for all $n\in\N$. 
%\begin{equation}
%\lim_{x\to\infty} \mathcal U_n(x) \frac{\sum\limits_{q\leq x}\#\unique(n,q)}{\sum\limits_{q\leq x}\#\lens(n,q)}
%=1.
%\end{equation}
\end{conjecture}

\begin{table}
\caption{Density of spectrally unique $(2n-1)$-dimensional lens spaces for $n=3$, $4$, $5$, $6$, $7$.}\label{table}
$
\begin{array}{ccccccc}
n&x & \sum\limits_{q\leq x}\#\unique(n,q)& \sum\limits_{q\leq x}\#\lens(n,q)&\mathcal U_n(x) \\ \hline \hline \rule{0pt}{14pt}
3
&50 & 40 & 990 & 0.95960 \\
&100 & 64 & 6680 & 0.99042 \\
&150 & 83 & 21881 & 0.99621 \\
&200 & 119 & 51580 & 0.99769 \\
&250 & 131 & 97546 & 0.99866 \\
&300 & 183 & 167856 & 0.99891 \\
\hline \rule{0pt}{14pt}
4
%&10 & 0 & 24 & 1.0000 \\
%&20 & 18 & 201 & 0.91045 \\
&30 & 47 & 693 & 0.93218 \\
%&40 & 78 & 1746 & 0.95533 \\
%&50 & 107 & 4275 & 0.97497 \\
&60 & 138 & 7966 & 0.98268 \\
%&70 & 168 & 13657 & 0.98770 \\
%&80 & 186 & 23897 & 0.99222 \\
&90 & 202 & 36699 & 0.99450 \\
%&100 & 218 & 51658 & 0.99578 \\
%&110 & 224 & 81401 & 0.99725 \\
&120 & 228 & 107094 & 0.99787 \\
%&130 & 254 & 144209 & 0.99824 \\
%&140 & 260 & 200764 & 0.99870 \\
&150 & 268 & 253189 & 0.99894 \\
\hline\rule{0pt}{14pt}
5
&10 & 0 & 28 & 1.0000 \\
&20 & 23 & 397 & 0.94207 \\
&30 & 74 & 1806 & 0.95903 \\
&40 & 127 & 5456 & 0.97672 \\
&50 & 197 & 17332 & 0.98863 \\
&60 & 255 & 37137 & 0.99313 \\
&70 & 345 & 71449 & 0.99517 \\
\hline \rule{0pt}{14pt}
6
&10 & 0 & 37 & 1.0000 \\
&20 & 14 & 801 & 0.98252 \\
&30 & 118 & 4640 & 0.97457 \\
&40 & 199 & 16497 & 0.98794 \\
&50 & 297 & 66751 & 0.99555 \\
&60 & 432 & 163935 & 0.99736 \\
\hline\rule{0pt}{14pt}
7
&10 & 0 & 41 & 1.0000 \\
&20 & 9 & 1501 & 0.99400 \\
&30 & 174 & 11188 & 0.98445 \\
&40 & 358 & 46750 & 0.99234 \\
&50 & 466 & 239345 & 0.99805 \\
\end{array}
$
\end{table}

%Tables~\ref{tablen=3}, \ref{tablen=4}, \ref{tablen=5}, \ref{tablen=6}, \ref{tablen=7} provide, 
Table~\ref{table} provides, for small values of $x$, numerical calculations of $\mathcal U_n(x)$ for $3\leq n\leq 7$.

A first step towards proving the above conjecture could be to show, for a fixed dimension $2n-1$, infinitely many values of $q\in\N$ such that no isospectrality can occur within $\lens(n,q)$, that is, $\unique(n,q)=\lens(n,q)$.

\begin{problem}\label{prob:conditionlensdisparate}
Provide conditions on $n$ and $q$ such that $\unique(n,q)=\lens(n,q)$; that is, for which no (non-trivial) isospectrality is possible among $(2n-1)$-dimensional lens spaces with volume $\frac{\vol(S^{2n-1})}{q}$. 
\end{problem}

\subsection{The length spectra of spherical space forms}

The length spectrum of a hyperbolic manifold, or more generally of a closed Riemannian manifold, is the set of lengths of closed geodesics on the manifold. The length spectrum has been an extremely fruitful area of research and is closely related to the Laplace eigenvalue spectrum. For example, the Selberg Trace Formula implies that two hyperbolic surfaces are (Laplace) isospectral if and only if they have the same length spectrum (see e.g.\ \cite[Chapter 7]{Buser92}). The work of Duistermaat and Guillemin and Duistermaat, Kolk, and Varadarajan shows that in the setting of compact locally symmetric manifolds of nonpositive curvature, the Laplace spectrum determines the length spectrum. (See also \cite[Theorem 10.1]{PR}.) 

\begin{question}
What is the relationship between the Laplace spectrum and the length spectrum of a spherical space form?
\end{question}

Some of the most interesting work concerning the length spectra of locally symmetric spaces of nonpositive curvature has been done in the arithmetic case and concerns the notion of {\it commensurability}. Recall that two Riemannian manifolds are said to be commensurable if they have a common finite degree covering space.  It is a result of Reid~\cite{Reid} that arithmetic hyperbolic surfaces with the same length spectra are necessarily commensurable. This result was extended to arithmetic hyperbolic $3$-manifolds by Chinburg, Hamilton, Long, and Reid~\cite{CHLR}, and to a very broad class of arithmetic locally symmetric spaces by Prasad and Rapinchuk~\cite{PR}. More recently, the second author, together with McReynolds, Pollack and Thompson, has proven (see \cite{LMPT}) that two incommensurable arithmetic hyperbolic manifolds of dimension $2$ or $3$ must have length spectra that disagree for some geodesic length bounded by an explicit function of the manifolds' volumes.

Although all spherical space forms are trivially commensurable (they are all covered with finite degree by the sphere $S^d$), one might instead consider the notion of {\it quotient commensurable} manifolds. Two Riemannian manifolds are said to be quotient commensurable if they share a common, finite degree quotient manifold. If $S^d/\Gamma_1$ and $S^d/\Gamma_2$ are spherical space forms, then they are quotient commensurable if and only if $\Gamma_1$ and $\Gamma_2$ both have finite index inside the group $G=\langle \Gamma_1, \Gamma_2\rangle$. 

\begin{question}
If two spherical space forms have the same length spectrum, must they be quotient commensurable?
\end{question}

\subsection{Isospectral but not strongly isospectral hyperbolic manifolds}

We end the article with a question that goes in the reverse direction in the sense that it concerns a result in the compact setting which is open in the non-compact one. 

We mentioned after Theorem~\ref{thm:Sunadamethod} that Sunada's method always produces \emph{strongly isospectral} manifolds. 
In particular, in addition to the Laplace-Beltrami operator (acting on $0$-forms), the corresponding Hodge-Laplace operators acting on $p$-forms are isospectral, for each $p$. 

It has been known for a while that Ikeda's examples of lens spaces cannot be constructed via Sunada's method (see \cite{Chen}). 
Miatello, Rossetti and the first author proved  that strongly isospectral lens spaces are necessarily isometric (see \cite[Prop.~7.2]{LMR-onenorm}), so isospectral and non-isometric lens spaces, including Ikeda's examples, are not strongly isospectral.

\begin{question}
Are there isospectral hyperbolic manifolds that are not strongly isospectral? 
\end{question} 

Furthermore, there are lens spaces $p$-isospectral for all $p$, and not strongly isospectral (see \cite{LMR-onenorm} and \cite{DeFordDoyle18}). 

\begin{question}
Are there hyperbolic manifolds $p$-isospectral for all $p$ that are not strongly isospectral? 
\end{question} 

Doyle and Rossetti have conjectured that the answer is negative (see \cite[\S9]{DR}).

\bibliographystyle{plain}

\end{document}